\documentclass[12pt, a4paper]{article}

\usepackage{pdfsync}
\usepackage{amsmath}\multlinegap=0pt
\usepackage[latin1]{inputenc}
\usepackage{amsthm}
\usepackage{amssymb}
\usepackage[francais,english]{babel}

\usepackage{hyperref}
\usepackage{graphicx}
\usepackage{blindtext}
\usepackage{enumitem}
\usepackage{mathtools}
\usepackage{color}

\numberwithin{equation}{section}
\theoremstyle{plain}
\newtheorem{thm}{Theorem}[section]

\newtheorem{coro}[thm]{Corollary}
\newtheorem{prop}[thm]{Proposition}
\newtheorem{lem}[thm]{Lemma}
\newtheorem{defi}[thm]{Definition}

\theoremstyle{definition}

\theoremstyle{remark}
\newtheorem{rem}[thm]{Remark}

\newcommand{\R}{\mathbb{R}}

\newcommand{\W}{\mathcal{W}}
\newcommand{\J}{\mathcal{J}}
\newcommand{\KK}{\mathcal{K}}

\newcommand\pref[1]{(\ref{#1})}

\let \eps\varepsilon

\newcommand\CC{{\cal C}}

\newcommand\A{{\cal A}}

\DeclareMathOperator{\argmax}{argmax}

\def\H {\operatorname{H}}

\newcommand\Ent{\mathop{\mathrm{Ent}}\nolimits}

\newcommand\dive{\mathrm{div}}

\newcommand\FP{\mathrm{FP}}

\def\<#1,#2>{\left<#1,#2\right>}
\let\bar\overline

\newcommand\tilR{\widetilde{R}}
\newcommand\tilQ{\widetilde{Q}}
\newcommand\tilmu{\widetilde{\mu}}

\newcommand\Leb{{\cal L}}

\newcommand\E{{\cal E}}

\def\PP{{\cal P}}
\def\SS{{\cal S}}
\def\CC{{\cal C}}

\newcommand{\TabThree}[1]{ %
\begin{tabular}{@{}c@{\hspace{0.5mm}}c@{\hspace{0.5mm}}c@{}}
 #1
\end{tabular}
}

\title{An entropy minimization approach to  second-order variational mean-field games}

\author{Jean-David Benamou \thanks{INRIA-Paris, MOKAPLAN,  rue Simone Iff, 75012, Paris, FRANCE and CEREMADE
\texttt{Jean-David.Benamou@inria.fr.}}
\and
Guillaume Carlier\thanks{CEREMADE, UMR CNRS 7534, Universit\'e Paris IX
Dauphine, Pl. de Lattre de Tassigny, 75775 Paris Cedex 16, FRANCE and INRIA-Paris, MOKAPLAN
\texttt{carlier@ceremade.dauphine.fr}}
\and
Simone Di Marino \thanks{Indam, Unit\`a SNS, Pisa, Italy
\texttt{simone.dimarino@altamatematica.it}}
\and 
Luca Nenna \thanks{CNRS and CEREMADE 
\texttt{nenna@ceremade.dauphine.fr}}.
}

\begin{document}

\maketitle

\begin{abstract}
We propose an entropy minimization viewpoint on variational mean-field games with diffusion and quadratic Hamiltonian.  We carefully analyze the time-discretization of such problems, establish $\Gamma$-convergence results as the time step vanishes and propose an efficient algorithm relying on this entropic interpretation as well as on the Sinkhorn scaling algorithm.

\end{abstract}

\textbf{Keywords:} Mean-Field Games, Fokker-Planck equation, $\Gamma$-convergence entropy minimization, Schr\"{o}dinger bridges, Sinkhorn algorithm.

\medskip

\textbf{MS Classification:} 65K10, 49M05.

 
\section{Introduction}\label{sec-intro} 

It is well-known since the seminal work of Lasry and Lions \cite{MFGLL1, MFGLL2, MFGLL3} that the mean-field game system 
\begin{equation}
\label{MFGsystem2}
\left\{\begin{array}{lll}
 -\partial_t u -\frac{1}{2} \Delta u  + \frac{1}{2} \vert \nabla u \vert^2= f[\rho_t], \; &(t,x)\in (0,T)\times \R^d\\
\partial_t \rho - \frac{1}{2} \Delta \rho -\dive(\rho \nabla u)=0, \; &(t,x)\in (0,T)\times \R^d \\
\rho{_{\vert}}_{t=0}=\rho_0, \; u{_{\vert}}_{t=T}= g[\rho_T],
\end{array}\right.\end{equation}
may be seen, at least formally (see \cite{CGPT} for a detailed analysis), as the system of optimality conditions for the minimization problem:
\begin{equation}\label{minimfg}
\inf_{(\rho, v)} \left\{ \J(\rho,v) \; : \partial_t \rho - \frac{1}{2} \Delta \rho + \dive(\rho v)=0, \rho{_{\vert}}_{t=0}=\rho_0\right\}
\end{equation}
where 
\begin{equation}\label{defdeJ}
\J(\rho,v):=  \frac{1}{2}  \int_0^T \int_{\R^d} \vert v_t\vert^2 \mbox{d} \rho_t(x) \mbox{d}t + \int_0^T F(\rho_t) \mbox{d} t + G(\rho_T).
\end{equation} 
We assume here that  $f$ and $g$ are \emph{potentials} i.e. admit $F$ and $G$ as primitives that is, $f[\rho]$ and $g[\rho]$ represent the directional derivative of the functionals $F$ and $G$ at $\rho$ in some suitable sense.

\smallskip

Our aim  is to relate precisely the \emph{Eulerian}   variational problem \pref{minimfg}-\pref{defdeJ} with a \emph{Lagrangian} relative entropy minimization problem and to develop a suitable efficient algorithm based on this entropic interpretation.  This entropy minimization  approach has its roots in the classical paper of Schr\"{o}dinger \cite{Schrodinger31}. The Schr\"{o}dinger bridge problem has deep connections with large deviations and has been extensively analyzed and developed by Mikami \cite{Mikami04} and  L\'eonard \cite{leonard2012schrodinger}, \cite{leonard2013survey} who in particular proved convergence of Schr\"{o}dinger bridges to optimal transport geodesics as the noise intensity vanishes. We also refer to Cattiaux and L\'eonard \cite{catleo} for further connections between entropy minimization, large deviations  and optimal control. 

\smallskip

Entropy minimization has also proved to be an efficient computational strategy for optimal transport by Cuturi \cite{Cut}  who made the connection with the powerful and versatile Sinkhorn scaling algorithm (also see \cite{Ben}). One advantage of reformulating \pref{minimfg}-\pref{defdeJ} as an entropy minimization is therefore that it enables one to use specific numerical schemes based on the Sinkhorn algorithm. However, the entropic viewpoint is, to the best of our knowledge, restricted to the quadratic Hamiltonian case and there are now efficient solvers, developed by Achdou and his coauthors \cite{achdou1, achdou2, achdou3} based on the PDE system which go much beyond the quadratic case. Gu\'eant in \cite{gueant} designed a monotone scheme specially intended for the quadratic Hamiltonian case, relying on ingenious changes of variables and the Hopf-Cole transform, the approach we propose in this paper shares some common features with Gu\'eant's method. For variational numerical methods based on more traditional convex optimization techniques, we refer for instance to  \cite{benamoucarlierjota}, \cite{andreev} and \cite{silva}. 

\smallskip

 The starting point of our analysis is the equivalence between the  Schr\"{o}dinger bridge problem and the optimal control (with kinetic energy  as cost)  of the Fokker-Planck equation as emphasized by  Chen, Georgiou and Pavon \cite{CGP} and Gentil, L\'eonard and Ripani \cite{GLR}.  The recent work of Gigli and Tamanini \cite{GigTam18a}  provides a very general and analytical approach to the dynamic formulation  of entropic transport problems. The entropy minimization viewpoint has been recently fruitfully used in the context of incompressible flows by Arnaudon et al. \cite{leonardBredinger} also see \cite{benamou2017generalized} for a numerical approach. Indeed, incompressibility (in the case of the flat torus instead of $\R^d$) can be seen as an instance of \pref{minimfg}-\pref{defdeJ} in the  somehow extreme case where $F(\rho)$ is $0$ when $\rho$ is uniform and $+\infty$ otherwise.

\smallskip 

The fact that variational problems like \pref{minimfg} can be reformulated as entropy minimization is not new and goes back to the analysis of Dawson and G\"{a}rtner \cite{DawsonGartner} on large deviations for weakly interacting diffusions. The connection between \pref{minimfg} and entropy minimization can be seen  as a consequence of Girsanov theory (under weak regularity assumptions, see L\'eonard \cite{LeonardGirsanov}), as explained in F\"{o}llmer \cite{Follmer} (see in particular section 1.4, p.165).  In the present work, we will follow a different and more elementary approach, motivated by numerical considerations.  We will show equivalence results at the level of discretized in time problems (which are precisely the ones we will numerically solve), prove $\Gamma$-convergence results as the discretization step vanishes and recover as a consequence equivalence for the continuous time minimization in \pref{minimfg} and an entropy minimization at the level of the path space.

\smallskip

The paper is organized as follows. Section \ref{sec-prel}  is devoted to some preliminaries and introduces an entropy minimization problem we propose as a Lagrangian counterpart of \eqref{minimfg}. In section \ref{sec-timedisc}, the time discretization of both the Lagrangian and Eulerian problems are shown to be equivalent. We then establish a $\Gamma$-convergence result which enables to recover the equivalence of the two formulations at the continuous time level as well in section \ref{sec-gamma}. In section \ref{sec-algo}, we propose a numerical scheme, based on the Sinkhorn scaling algorithm to solve the entropy minimization problem and we present some numerical experiments.

\smallskip

\textbf{Acknowledgments.} We want to thank the anonymus referees for their observations and careful reading of the paper. Part of this work was done while D.M. was a FSMP visitor researcher in INRIA.

\section{Preliminaries}\label{sec-prel}

We denote by $\PP_2(\R^d)$  the set of Borel probability measures with finite second moments, endowed with the Wasserstein distance $\W_2$, whose square is defined by
\[\W_2^2(\rho_0, \rho_1):=\inf_{\gamma\in \Pi(\rho_0, \rho_1)}  \Big\{ \int_{\R^d\times \R^d} \vert x-y\vert^2 \mbox{d} \gamma(x,y)    \Big\}, \; (\rho_0, \rho_1) \in \PP_2(\R^d)\times \PP_2(\R^d)\]
$\Pi(\rho_0, \rho_1)$ being the set of transport plans between $\rho_0$ and $\rho_1$ i.e. the set of probability measures on $\R^d\times \R^d$ having $\rho_0$ and $\rho_1$ as marginals. 

\smallskip
We denote by $\Leb^d$ the Lebesgue measure on $\R^d$ and whenever a measure $\rho$ is absolutely continuous with respect to $\Leb^d$, we will denote by $\rho$ both the measure itself and its density with respect to $\Leb^d$. Throughout the paper we will assume that $F$ and $G$ can be written as the sum of a convex local term and a term which is continuous for the Wasserstein distance $\W_2$ i.e. 
\[F=F_1+F_2, \; G=G_1+G_2,  \]
such that
\begin{itemize}
\item there are convex continuous integrands $L_F$ and $L_G$ : $\R^d \to \R_+$ such that $L_F(0)=L_G(0)=0$ and
\[\begin{split}
F_1(\rho)&=\begin{cases} \int_{\R^d} L_F(\rho(x)) \mbox{d} x  \mbox{ if $\rho \ll \Leb^d$}\\ + \infty \mbox{ otherwise}   \end{cases},\\
G_1(\rho)&=\begin{cases} \int_{\R^d} L_G(\rho(x)) \mbox{d} x  \mbox{ if $\rho \ll \Leb^d$}\\ + \infty \mbox{ otherwise} \end{cases}
\end{split}\]
\item $F_2$,  $G_2$  :   $\PP_2(\R^d)\to \R_+$ are  continuous for $\W_2$. 
\end{itemize}

\smallskip

Given a Polish space $X$,  a Borel probability measure $q$ on $X$ and  $r$ a $\sigma$-finite measure on $X$,  the relative entropy of $q$ with respect to $r$  is given\footnote{some caution must be taken when $r$ is $\sigma$-finite but unbounded,  in this case one can find a measurable and bounded from below potential $V$: $X\mapsto \R$ such that $e^{-V} r$ is a probability measure on $X$ and then define for each probability $q$ for which $\int_X V \mbox{d}q <+\infty$ $H(q\vert r)$ as $H(q\vert r)=H(q\vert e^{-V} r)-\int_X V \mbox{d} q$ see Appendix 1 in \cite{leonard2013survey} for details. When $R$ is the reversible Wiener measure, we shall apply this trick with $V(\omega)=\pi \vert \omega(0)\vert^2$ which will enable us to define $H(Q\vert R)$ as soon as $Q_0\in \PP_2(\R^d)$.} by 
\[\H(q\vert r):=\begin{cases} \int_X \log\Big( \frac{ d q}{d r } \Big) \mbox{d} q \mbox{ if $q \ll r $}\\ + \infty \mbox{ otherwise}   \end{cases}\]
where $\frac{ d q}{d r} $ stands for the Radon-Nikodym derivative of $q$ with respect to $r$. If $X=\R^d$ and $\rho\in \PP_2(\R^d)$ we shall simply denote by $\Ent(\rho)$ the relative entropy of $\rho$ with respect to $\Leb^d$. Let us also recall that $\Ent $ is controlled from below   by the second moment (see \pref{entrminus-bound} for a linear lower bound).
We shall always assume that the initial condition $\rho_0$ satisfies
\begin{equation}\label{hyprho0}
\rho_0 \in \PP_2(\R^d) \cap L^1(\R^d), \; \mbox{ and } \Ent(\rho_0) <+\infty
\end{equation}
so that $-\infty <\Ent(\rho_0)<+\infty$.  

\smallskip

As outlined in the introduction, our goal is to relate \pref{minimfg} with a relative entropy minimization over the path space $\Omega:=C([0,T], \R^d)$. Let $e_t$ denote the evaluation at time $t\in [0,T]$ i.e. $e_t(\omega)=\omega(t)$ for every $\omega\in \Omega$. We take as reference measure on $\Omega$ the reversible Wiener measure $R$ defined by
\begin{equation}
R:=\int_{\R^d} \delta_{x+ B} \, \mbox{d} x
\end{equation}
where $B=(B_t)_{t\in [0,T]}$ is a standard $d$-dimensional Brownian motion starting at $0$. Given $Q\in \PP(\Omega)$ a Borel probability measure on $\Omega$, set
\begin{equation}\label{defmargQ}
Q_t:={e_t}_\# Q, \; Q_{s, t}:=(e_{t}, e_{s})_\# Q, \; 0\le s < t \le  T,
\end{equation}
and more generally if $t_1, \cdots, t_n \in [0,T]^n$
\[Q_{t_1, \cdots, t_n}:=(e_{t_1}, \cdots, e_{t_n})_\# Q.\]
We define similarly  $R_t$, $R_{s,t}$ and $R_{t_1, \cdots, t_n}$ and  observe that by construction $R_t:= {e_t}_\# R=\Leb^d$ for every $t\in [0,T]$ and for every $0\le s < t \le  T$, $R_{s,t}=(e_{t}, e_{s})_\# R$ is given by
\[R_{s,t} (\mbox{d}x, \mbox{d}y )= P_{t-s}(x-y)\mbox{d}x \mbox{d}y\]
where $P_t$ is the heat kernel\footnote{Note that our convention is to take the heat kernel associated with $\frac{1}{2}\Delta$.}:
 \begin{equation}\label{heatker}
 P_t(z):=\frac{1}{(2\pi t)^{\frac{d}{2}} }\exp\Big(-\frac{\vert z \vert^2}{2t}\Big), \; t>0, \; x\in \R^d.
\end{equation}

Let us finally introduce a few notations. Given $\mu, \nu \in \PP_2(\R^d)$ and $h\in (0,T)$ let us define $\FP_h(\mu,\nu)$ as the infimum of the kinetic energy
\begin{equation}
\FP_h(\mu,\nu):=\inf_{(\rho, v)} \Big\{ \frac{1}{2}\int_0^h \int_{\R^d}  \vert v_t\vert^2 \mbox{d} \rho_t(x) \mbox{d}t \;  \Big\}
\end{equation} 
among pairs $(\rho, v)$ solving the Fokker-Planck equation on $(0,h)\times \R^d$ with endpoints $\mu$ and $\nu$
\[\partial_t \rho -\frac{1}{2} \Delta \rho+\dive(\rho v)=0, \; \rho_{\vert_{t=0}}=\mu, \; \rho_{\vert_{t=h}}=\nu.\]
Also define
\begin{equation}\label{schrodyn}
S_h(\mu, \nu):=\inf \{ H(Q\vert R) \;: \; Q\in \PP(C([0,h], \R^d), \;  Q_0=\mu, \; Q_h=\nu\}
\end{equation}
(where, slightly abusing notations in the formula above, $R$ denotes the reversible Wiener measure on $C([0,h], \R^d)$). Following L\'eonard \cite{leonard2013survey}, and writing the disintegrations of $Q$ and $R$ with respect to their marginals $Q_{0,h}$, $R_{0, h}$:
\[Q=\int_{\R^d \times \R^d} Q(.\vert x_0, x_h) \mbox{d} Q_{0,h}(x_0, x_h), \; R=\int_{\R^d \times \R^d} R(.\vert x_0, x_h) \mbox{d} R_{0,h}(x_0, x_h)\]
(so that $R(.\vert x_0, x_h)$ is the probability Law of a Brownian bridge)
since 
\[H(Q\vert R)=H(Q_{0,h} \vert R_{0,h})+ \int_{\R^d\times \R^d}  H(Q(. \vert  x_0, x_h)\vert R( . \vert x_0, x_h) )\mbox{d} Q_{0,h}(x_0, x_h)\]
and $H(Q(. \vert  x_0, x_h)\vert R( . \vert x_0, x_h) )\ge 0$ with an equality when $Q(. \vert  x_0, x_h)= R( . \vert x_0, x_h)$ (which means that $R$ and the optimal $Q$ share the same bridge) one can express $S_h$ as the value of a static entropy minimization problem:
\begin{equation}\label{schrosta}
S_h(\mu, \nu)=\inf \{ H(\gamma \vert R_{0,h}) \; : \gamma \in \Pi(\mu, \nu)\}
\end{equation}
and the solution of the dynamic problem \pref{schrodyn} is obtained by 
\[Q=\int_{\R^d \times \R^d} R(.\vert x_0, x_h) \mbox{d} \gamma(x_0, x_h)\]
with $\gamma$ optimal for the static problem \pref{schrosta}. 

The following result, proven in  Chen, Georgiou, Pavon \cite{CGP}  and Gentil, L\'eonard and Ripani \cite{GLR} (in a more general setting) connects $\FP_h$ to $S_h$ and can be viewed as a noisy version of the Benamou-Brenier formula:
 
\begin{thm}\label{noisybb}
Let $\mu, \nu \in \PP_2(\R^d)$ and $h\in (0,T)$, $\FP_h(\mu, \nu)$ and $S_h(\mu, \nu)$ are related by
\begin{equation}\label{lienfps}
S_h(\mu, \nu)= \FP_h(\mu, \nu)+ \Ent(\mu).
\end{equation}
\end{thm}

Let us also mention an alternative formulation, involving the Fisher information (see \cite{GLR, CGP}):
\begin{equation}\label{schrofisher}
 \FP_h(\mu, \nu)=\inf_{(\rho, v)} \Big\{\int_0^h \Big(\int_{\R^d}   \frac{1}{2}\vert v_t\vert^2 \mbox{d} \rho_t(x) +\frac{1}{8} I(\rho_t)\Big) \mbox{d}t   \Big\} + \frac{1}{2} (\Ent(\nu)-\Ent(\mu))
\end{equation}
where the infimum is taken among solutions\footnote{Note that if $\int_0^T \int_{\R^d} \vert v_t\vert^2 \mbox{d} \rho_t(x) \mbox{d}t <+\infty$ and $\partial_t \rho +\dive(\rho v)=0$ then $t\mapsto \rho_t$ is continuous in $\W_2$.} of the continuity equation with fixed endpoints
\[\partial_t \rho +\dive(\rho v)=0, \; \rho_{\vert_{t=0}}=\mu, \; \rho_{\vert_{t=h}}=\nu\]
and $I$ denotes the Fisher information
\begin{equation}\label{fisherinfo}
I(\rho):= \int_{\R^d} \vert \nabla \log(\rho)\vert^2 \rho=\int_{\R^d} \frac{\vert \nabla \rho\vert^2}{ \rho} =4 \Vert \nabla (\sqrt{\rho}) \Vert^2_{L^2(\R^d)}.
\end{equation}

Given a continuous curve of measures $\mu\in C([0,T], (\PP_2(\R^d), \W_2))$ :  $t\in [0,T]\mapsto \mu_t \in \PP_2(\R^d)$ define
\[\E(\mu):=\inf_{v} \Big\{ \frac{1}{2}\int_0^T \int_{\R^d}  \vert v_t\vert^2 \mbox{d}\mu_t(x) \mbox{d}t \;  : \; \partial_t \mu -\frac{1}{2} \Delta \mu+\dive(\mu v)=0\Big\} \]
as well as the cost
\[\CC(\mu):= G(\mu_T)+ \int_0^T F(\mu_t) \mbox{d} t \]
so that the variational formulation \pref{minimfg}-\pref{defdeJ} of the MFG system can be simply rewritten as
\begin{equation}\label{minimfgsynt}
\inf\{ \E(\mu)+\CC(\mu) \;  : \; \mu \in C([0,T], (\PP_2(\R^d), \W_2)), \; \mu_{\vert_{t=0}}=\rho_0\}
\end{equation}
For $\mu=(\mu_t)_{t\in [0,T]} \in C([0,T], (\PP_2(\R^d), \W_2)$, define also the minimal entropic cost
\begin{equation}\label{schroinfmarg}
\SS(\mu):=\inf\{ H(Q\vert R) \; : \; Q_t=\mu_t, \; \forall t\in [0,T]\}
\end{equation}
which can be viewed as a variant of \pref{schrodyn} with infinitely many marginal constraints. We shall prove, by a careful inspection of a suitable time discretization (section \ref{sec-timedisc}) and $\Gamma$-convergence arguments (see section \ref{sec-gamma}), that when $\Ent(\mu_0)<+\infty$ the following relation holds
\[\SS(\mu)=\E(\mu) +\Ent(\mu_0)\]
so that \pref{minimfgsynt} can be reformulated as
\begin{equation}\label{minentrmfg}
\inf \Big\{ H(Q\vert R)+ \CC((Q_t)_{t\in [0,T]}) \; : \; Q\in \PP(\Omega), \; Q_0=\rho_0\Big\}. 
\end{equation}

\section{Time discretization}\label{sec-timedisc}

\subsection{Discretization}

Given a positive integer $N$ and $\mu_0, \cdots, \mu_N\in \PP_2(\R^d)$, define the respective discretization of $\E$, $\CC$ and $\SS$:
\begin{equation}\label{discreteFP}
\E^N(\mu_0, \cdots , \mu_N):=\sum_{k=0}^{N-1} \FP_{\frac{T}{N}} (\mu_k, \mu_{k+1}), \;
\end{equation}
\begin{equation}\label{discreteC}
\CC^N(\mu_0, \cdots , \mu_N):=\frac{T}{N} \sum_{k=1}^{N-1} F(\mu_k) + G(\mu_N)
\end{equation}
and
\begin{equation}\label{discrschro}
\SS^N(\mu_0, \cdots, \mu_N):=\inf\Big\{ H(Q\vert R) \; : \; Q\in \PP(\Omega), \;  Q_{\frac{kT}{N}}=\mu_k, \; k=0, \cdots N\Big\}
\end{equation}
and observe that, exactly as for \pref{schrodyn}, $\SS^N$ can be written as the value of the static multi-marginal problem
\begin{equation}\label{discrshrosta}
\SS^N(\mu_0, \cdots, \mu_N):=\inf_{\gamma^N\in \Pi(\mu_0, \cdots, \mu_N)} H(\gamma^N \vert R^N)
\end{equation}
where $\Pi(\mu_0, \cdots, \mu_N)$ is the set of probability measures on $(\R^d)^{N+1}$ having $\mu_0, \cdots, \mu_N$ as marginals and 
\begin{equation}\label{discretewienerr}
R^N:=R_{0, \frac{T}{N}, \cdots, T}.
\end{equation}
One then discretizes \pref{minimfgsynt} as
\begin{equation}\label{minimfgsyntn}
\inf\{ \E^N(\mu_0, \cdots, \mu_N)+\CC^N(\mu_0, \cdots, \mu_N), \; \mu_i \in \PP_2(\R^d), \;  \mu_0=\rho_0\}. 
\end{equation}
In an analogous way, the discrete counterpart of \eqref{minentrmfg} becomes
 \begin{equation}\label{entminimfgsyntn}
\inf\{ \SS^N(\mu_0, \cdots, \mu_N)+\CC^N(\mu_0, \cdots, \mu_N), \; \mu_i \in \PP_2(\R^d), \;  \mu_0=\rho_0\}. 
\end{equation}

\subsection{On entropy minimization and Markovianity}\label{sec-markov}

In this section, our aim is to emphasize the role of Markovianity in entropy minimization problems with fixed marginals as in \pref{discrschro}. First, we need to recall some fundamentals regarding disintegrations and conditional laws. Given $Y_1$ and $Y_2$, two Polish space and a Borel probability measure $P\in  \mathcal{P}(Y_1\times Y_2)$ on the product and denoting by $P_1$ the first marginal of $P$, the disintegration Theorem (see \cite{DellacherieMeyer}) says that there exists a ($P_1$-a.e. unique) Borel family of probabilty measure $y_1 \mapsto P(. \vert y_1)\in \mathcal{P}(Y_2)$ such that
\begin{equation}\label{disintP1}
P(\mbox{d} y_1 \mbox{d} y_2)=P_1(\mbox{d} y_1) P (\mbox{d} y_2\vert y_1) 
\end{equation}  
\pref{disintP1} is called the disintegration of $P$ with respect to its first marginal $P_1$ and means that for every $\varphi\in C_b(Y_1\times Y_2)$ one has
\[\int_{Y_1\times Y_2} \varphi \mbox{d}P =\int_{Y_1} \Big(\int_{Y_2} \varphi(y_1, y_2) \mbox{d}P(y_2 \vert y_1) \Big) \mbox{d} P_1(y_1).\]
Of course $P(. \vert y_1)$ can be interpreted as the conditional probability of $y_2$ given $y_1$. In a similar way, $P$ can be disintegrated with respect to its second marginal as
\begin{equation}\label{disintP2}
P(\mbox{d} y_1 \mbox{d} y_2)=P (\mbox{d} y_1 \vert y_2) P_2 (\mbox{d} y_2). 
\end{equation}

Given $n\geq 3$, $n$ Polish spaces $X_1, \ldots X_n$, $X:=X_1 \times X_2 \times \cdots \times X_n$ and $Q\in \mathcal{P}(X)$, denoting by $e_i$ the $i$-th canonical projection from $X$ to $X_i$, we define
\[Q_i:={e_i}_\#Q, \; Q_{i, \ldots, k}:=(e_i, \ldots, e_k)_\#Q, \;  \mbox{ for } 1\leq i<k\leq n.\]
Given $i, k$ with $1\leq k <i\leq  n$, the disintegration theorem recalled above (with $Y_1=X_1 \times \cdots X_{k}$ and $Y_2=X_{k+1}\times \ldots X_i$) gives the (past, and future conditioned)  
disintegrations:
\[\begin{split}
Q_{1, \ldots, i}(\mbox{d} x_1  \ldots \mbox{d} x_i)= Q_{1, \ldots, k}(\mbox{d} x_1  \ldots \mbox{d} x_{k}) Q_{1, \ldots, i} (\mbox{d} x_{k+1} \ldots \mbox{d} x_i \vert x_1, \ldots, x_{k})\\
=  Q_{1, \ldots, i} (\mbox{d} x_{1} \ldots \mbox{d} x_k \vert x_{k+1}, \ldots, x_{i}) Q_{k+1, \ldots, i} (\mbox{d} x_{k+1} \ldots \mbox{d} x_i).
\end{split}\]
\begin{defi}
Let $n\geq 3$,  $X_1, \cdots, X_n$ be Polish spaces, $X:=X_1 \times X_2 \times \cdots \times X_n$, a Borel probability measure on $X$, $Q$,   is called Markovian if for every $i$ with $3\leq i \leq n$ one has
\[ Q_{1, \ldots, i} (\mbox{d} x_{i}  \vert x_1, \ldots, x_{i-1})=Q_{i-1, i} (\mbox{d} x_i \vert x_{i-1}).\]
\end{defi}
Interpreting  $i$ as a discrete time index, the previous definition agrees with the more standard notion of Markov process in the sense that $Q$ is a Markovian measure if it is the law of a Markov process. Note also that if $Q$ is Markovian so are $Q_{1, \ldots, n-1}$ and $Q_{2, \ldots, n}$.

The discretized reversible Wiener measure measure $R^N$ defined  by \pref{discretewienerr} is not a probability measure (it is not even finite but just $\sigma$-finite) however the equivalent measure  on $({\R^d})^{N+1}$
\begin{equation}\label{tilRNw}
\tilR^N(\mbox{d}x_0 \ldots \mbox{d}x_N):=e^{-\pi \vert x_0\vert^2} R^N(\mbox{d}x_0 \ldots \mbox{d}x_N),
\end{equation}
is a Markov probability measure. 
Let us now define Markov concatenations (or glued plans):

\begin{defi} Given Polish spaces $Y_1$, $Y_2$, $Y_3$,  two Borel probability measures $P_{12} \in \mathcal{P}(Y_1 \times Y_2)$ and $P_{23} \in \mathcal{P}(Y_2 \times Y_3)$ having the same marginal $P_2$ on $Y_2$,    the Markov-concatenation of the two plans $P_{12}$ and $P_{23}$, which we denote  by $P_{12} \circ P_{23}$ is the Borel probability measure on $Y_1\times Y_2\times Y_3$ defined by
\begin{equation}\label{eqn:defconcatenation}(P_{12} \circ P_{23})(d y_1 \mbox{d}y_2  \mbox{d}y_3)= P_{12}(\mbox{d} y_1 \vert y_2) P_2(\mbox{d} y_2) P_{23} (\mbox{d} y_3 \vert y_2)
\end{equation}
\end{defi}

Note that $(e_1, e_2)_\# (P_{12} \circ P_{23})=P_{12}$, $(e_2, e_3)_\# (P_{12} \circ P_{23})=P_{23}$. Note also that $P_{12} \circ P_{23}$ can be written equivalently as 
\[\begin{split}
(P_{12}\circ P_{23})(\mbox{d}y_1 \mbox{d}y_2  \mbox{d}y_3)&= P_{12}(\mbox{d} y_1 \mbox{d} y_2 )  P_{23} (\mbox{d} y_3 \vert y_2)\\
= P_{12}(\mbox{d} y_1 \vert  y_2 )  P_{23} (\mbox{d} y_2 \mbox{d}  y_3)
\end{split}\]
which obviously implies that  $(P_{12} \circ P_{23})$ is Markovian on $Y_1\times Y_2 \times Y_3$. 

Conversely, the same calculation proves also that if $Q$ is Markovian and $(e_i, e_{j})_{\sharp}Q=Q_{ij}$, then $Q=Q_{12} \circ Q_{23}$. In fact, by definition of Markovian measure, we have
$$Q(\mbox{d}y_1 \mbox{d}y_2  \mbox{d}y_3) = Q_{12}(\mbox{d}y_1 \mbox{d}y_2) Q_{23} ( \mbox{d} y_3 | y_2); $$
then applying the disintegration theorem to $Q_{12}$ gives us exactly the representation given in \eqref{eqn:defconcatenation}, proving the claim.

If now $n\geq 3$ and $Q\in \PP(X)$ with $X=X_1 \times \cdots \times X_n$ as above, for every pair of indices $1\leq k <i \leq n$ since $Q_{1, \ldots, k}$ and $Q_{k, \ldots, i}$ share the same marginal on $X_k$, one can define the concatenation $Q_{1, \ldots, k} \circ Q_{k, \ldots, i}$. The characterization of Markov measures in terms of concatenations is given by:

\begin{lem}\label{markovfromgluing}
Let $n\geq 3$,  $X_1, \cdots, X_n$ be Polish spaces, $X:=X_1 \times X_2 \times \cdots \times X_n$  and $Q$ be a Borel probability measure on $X$ then the following statements are equivalent

\begin{enumerate}

\item $Q$ is Markovian,

\item $Q_{1, \ldots, n-1}$ is Markovian and $Q=Q_{1, \ldots, n-1} \circ Q_{n-1, n}$,

\item  $Q_{2, \ldots, n}$ is Markovian and $Q=Q_{1,2} \circ Q_{2, \ldots, n}$.

\end{enumerate}

\end{lem}

\begin{proof}
$1.$ obviously implies implies $2.$. Assume now that $2$. holds and let us show that $Q_{2, \ldots, n}$ is Markovian, first note that $Q_{2, \ldots, n-1}$ is Markovian since $Q_{1, \ldots, n-1}$ is and  since $Q=Q_{1, \ldots, n-1} \circ Q_{n-1, n}$ we also have $Q_{2, \ldots, n} =Q_{2, \ldots, n-1} \circ Q_{n-1, n}$ which by uniqueness of the disintegration proves that $Q_{2, \ldots, n}(\mbox{d} x_n \vert x_1, \ldots, x_{n-1})=Q_{n-1, n} (\mbox{d} x_n \vert x_{n-1})$, this proves the Markovianity of $Q_{2, \ldots, n}$. Now observe
\[\begin{split}
Q&=Q_{12}(\mbox{d} x_1 \mbox{d} x_2) Q(\mbox{d} x_3 \ldots \mbox{d} x_n \vert x_1, x_2)\\
&= Q_{12}(\mbox{d} x_1 \mbox{d} x_2) Q_{1, \ldots, n-1} (\mbox{d}x_3\ldots \mbox{d} x_{n-1}\vert x_1, x_2) Q(\mbox{d}x_n \vert x_1, \ldots, x_{n-1})\\
&= Q_{12}(\mbox{d} x_1 \mbox{d} x_2) Q_{1, \ldots, n-1} (\mbox{d}x_3\ldots dx_{n-1}\vert x_2) Q_{n-1, n} (\mbox{d} x_n \vert x_{n-1})\\
&=Q_{12}(\mbox{d} x_1 \mbox{d} x_2) Q_{2, \ldots, n-1} (\mbox{d}x_3\ldots \mbox{d}x_{n-1}\vert x_2) Q_{2, \ldots, n} (\mbox{d} x_n \vert x_2, \ldots, x_{n-1})\\
&=Q_{12} (\mbox{d} x_1 \mbox{d} x_2) Q_{2, \ldots, n}( \mbox{d} x_3 \ldots \mbox{d}x_{n}\vert x_2)=Q_{1,2} \circ Q_{2, \ldots, n}
\end{split}\]
where we have used that $Q_{1, \cdots, n-1}$ is Markov and $Q=Q_{1, \ldots, n-1} \circ Q_{n-1, n}$ in the third line, and the Markovianity of $Q_{2, \ldots, n}$ in the fourth line to replace $ Q_{n-1, n} (\mbox{d} x_n \vert x_{n-1})=Q_{2, \ldots, n} (\mbox{d} x_n \vert x_{n-1})$ with $Q_{2, \ldots, n} (\mbox{d} x_n \vert x_2, \ldots, x_{n-1})$. 

Finally, assume that $3$. holds, and let us fix $3\leq i \leq n$;  thanks to the Markovianity of $Q_{2, \ldots, n}$, we have
\[\begin{split}
Q_{1, \ldots, i} (\mbox{d} x_1 \ldots \mbox{d} x_i)&=Q_{12}(\mbox{d} x_1 \vert x_2) Q_{2, \ldots, i}(\mbox{d} x_2 \ldots \mbox{d} x_i)\\ 
&=Q_{12}(\mbox{d} x_1 \vert x_2) Q_{2, \ldots, i-1}(\mbox{d} x_2 \ldots \mbox{d} x_{i-1}) Q_{i-1, i}(\mbox{d} x_i \vert x_{i-1})
\end{split}\]
which yields, again by uniqueness of disintegrations, that $Q_{1, \ldots, i} (\mbox{d} x_i \vert x_{i-1})=Q_{i-1, i}(\mbox{d} x_i \vert x_{i-1})$ so that $Q$ is Markovian. 
\end{proof}
An immediate consequence of the previous Lemma is that when $Q$ is Markovian then 
\[Q_{1, \ldots, 4}= Q_{1,2} \circ (Q_{2,3} \circ Q_{3, 4})=(Q_{1,2} \circ Q_{2,3})\circ Q_{3,4}\]
so that the successive concatenation operations are associative, hence we shall simply write, without parentheses
\[Q=Q_{1,2} \circ \cdots \circ Q_{n-1, n}.\]

\begin{lem}\label{lemmarkov} Given the product of Polish  spaces $X= X_1 \times X_2 \times \cdots \times X_n$, probability measures $P \in \mathcal{P} ( X) $ and $R \in \mathcal{M}(P)$, let us denote $P_i= (e_i)_{\#}P$, $R_i= (e_i)_{\#}R$, $P_{i, i+1} = (e_i, e_{i+1} )_{\#} P$ and $R_{i, i+1} = (e_i, e_{i+1} )_{\#} R$. Let us suppose that $R$ is Markovian so that $R=R_{1,2}\circ R_{2,3} \circ \cdots \circ R_{n-1,n}$, then we have
  $$H(P|R) \geq \sum_{i=1}^{n-1}  H(P_{i,i+1}|R_{i,i+1}) - \sum_{i=2}^{n-1} H(P_{i}|R_{i}),$$
with equality if and only  i.e. $P$ is itself Markov so that if $P=P_{1,2} \circ P_{2,3} \circ \ldots \circ P_{n-1,n}$.
\end{lem}

\begin{proof} We will prove the lemma by induction on $n$, starting with the first nontrivial case $n=3$. For $n=3$, we have
\[H(P_{12}\vert R_{12})=H(P_2\vert R_2)+ \int_{X_2} H(P_{12}(  \mbox{d} x_1 \vert x_2) \vert R_{12}(  \mbox{d} x_1\vert x_2)) \mbox{d} P_2(x_2)\]
and 
\[H(P_{23}\vert R_{23})=H(P_2\vert R_2)+ \int_{X_2} H(P_{23}( \mbox{d} x_3\vert x_2) \vert R_{23}(\mbox{d} x_3\vert x_2)) \mbox{d} P_2(x_2)\]
but
\[H(P\vert R)=H(P_2\vert R_2)+ \int_{X_2} H(P( \mbox{d} x_1 \mbox{d} x_3   \vert x_2) \vert R( \mbox{d} x_1 \mbox{d} x_3  \vert x_2) )\mbox{d} P_2(x_2).\]
Let us recall that it follows from the strict convexity of $t\in \R_+ \mapsto t\log(t)$ that if $\gamma$ is a probability measure on a product space with marginals $\gamma_1$ and $\gamma_2$ then $H(\gamma\vert \mu_1 \otimes \mu_2)\ge H(\gamma_1 \vert \mu_1)+H(\gamma_2 \vert \mu_2)$ with equality exactly when $\gamma$ is a product measure $\gamma =\gamma_1 \otimes \gamma_2$.  
Then, observing that for $P_2$ almost every $x_2$, the conditional probability,  $P( \mbox{d} x_1 \mbox{d} x_3 \vert x_2)$ has marginals $P_{12}(  \mbox{d} x_1 \vert x_2)$ and $P_{23}( \mbox{d} x_3 \vert x_2)$ and that since $R$ is Markov $R( \mbox{d} x_1 \mbox{d} x_3 \vert x_2)=R_{12}( \mbox{d} x_1\vert x_2) \otimes  R_{23}(  \mbox{d} x_3 \vert x_2)$, we have 
\[ \begin{split}H(P(  \mbox{d} x_1 \mbox{d} x_3  \vert x_2) \vert R(  \mbox{d} x_1 \mbox{d} x_3\vert x_2) )\ge H(P_{12}(  \mbox{d} x_1 \vert x_2) \vert R_{12}( \mbox{d} x_1 \vert x_2)) \\+  H(P_{23}( \mbox{d} x_3\vert x_2) \vert R_{23}(\mbox{d} x_3\vert x_2))
\end{split}\]
with an equality if and only if $P( \mbox{d} x_1 \mbox{d} x_3 \vert x_2)=P_{12}( \mbox{d} x_1 \vert x_2) P_{23}( \mbox{d} x_3\vert x_2)$ i.e. $P$ is Markov. This proves the claim for $n=3$. Assume the claim is satisfied for $n$ and consider $P \in \mathcal{P} ( X_1\times \cdots \times X_{n+1})$ and $R$ a Markov probability on $X_1\times \cdots \times X_{n+1}$, we then have 
\[\begin{split}
H(P\vert R)&=H(P_{1, \cdots, n} \vert R_{1, \cdots, n})+\\
& \int_{X_1\times \cdots \times X_n} H(P(\mbox{d} x_{n+1} \vert x_1, \cdots x_n)\vert R(\mbox{d} x_{n+1}\vert x_1, \cdots x_n)) \mbox{d} P_{1, \cdots, n} (x_1, \cdots, x_n).
\end{split}\]
From the validity of the claim for $P_{1, \cdots, n}$ and $R_{1, \cdots, n}$, we have
\begin{equation}\label{casn}
H(P_{1, \cdots, n} \vert R_{1, \cdots, n}) \ge  \sum_{i=1}^{n-1}  H(P_{i,i+1}|R_{i,i+1}) - \sum_{i=2}^{n-1} H(P_{i}|R_{i}).
\end{equation}
So  we are left to show that
\[\begin{split}
&\int_{X_1\times \cdots \times X_n} H(P( \mbox{d} x_{n+1} \vert x_1, \cdots x_n)\vert R(\mbox{d} x_{n+1} \vert x_1, \cdots x_n)) \mbox{d} P_{1, \cdots, n} (x_1, \cdots, x_n)\\
&\ge H(P_{n, n+1}\vert R_{n, n+1}) -H(P_n \vert R_n)\\
&=\int_{X_n} H(P_{n, n+1}( \mbox{d} x_{n+1}  \vert x_n)\vert R_{n, n+1}( \mbox{d} x_{n+1}  \vert x_n)) \mbox{d} P_n(x_n).
\end{split}\]
Now we observe that since $R$ is Markov $R(\mbox{d} x_{n+1}  \vert x_1, \cdots x_n)= R_{n,n+1}(\mbox{d} x_{n+1}  \vert x_n)$,  we then write
\[\begin{split}
&\int_{X_1\times \cdots \times X_n} H(P( \mbox{d} x_{n+1} \vert x_1, \cdots x_n)\vert R_{n,n+1}(\mbox{d} x_{n+1} \vert  x_n) \mbox{d} P_{1, \cdots, n} (x_1, \cdots, x_n)\\
&= \int_{X_n} \Big( \int_{X_1\times \cdots \times X_{n-1}} H(P( \mbox{d} x_{n+1} \vert x_1, \cdots x_n)\vert R(\mbox{d} x_{n+1} \vert x_n) ) \mbox{d} P_{1, \cdots, n} (x_1, \cdots, x_{n-1} \vert x_n) \Big) \mbox{d} P_n(x_n)
\end{split}\]
using the convexity of $H(. \vert R(\mbox{d} x_{n+1} \vert x_n))$ and the fact that
\[\int_{X_1 \times \cdots \times X_{n-1}}  P(\mbox{d} x_{n+1} \vert x_1\cdots, x_n) \mbox{d} P_{1, \cdots, n} (x_1, \cdots, x_{n-1} \vert x_n) =P_{n, n+1}(\mbox{d}x_{n+1} \vert x_n)\]
we deduce
\[\begin{split}
  \int_{X_1\times \cdots \times X_{n-1}} H(P( \mbox{d} x_{n+1} \vert x_1, \cdots x_n)\vert R(\mbox{d} x_{n+1} \vert x_n)) \mbox{d} P_{1, \cdots, n} (x_1, \cdots, x_{n-1} \vert x_n)\\
 \ge H(P_{n, n+1}( \mbox{d} x_{n+1}  \vert x_n)\vert R_{n, n+1}( \mbox{d} x_{n+1}  \vert x_n))
\end{split}\]
integrating with respect to $\mbox{d} P_n(x_n)$ gives the desired inequality. Now in the equality case, there should be an equality in \pref{casn} so that $P_{1, \cdots, n}$ should be Markov, but there should also be an inequality in the convexity inequality for the relative entropy above which implies that $P(\mbox{d} x_{n+1} \vert x_1\cdots, x_n)=P_{n, n+1}(\mbox{d} x_{n+1} \vert x_n)$ and these two conditions imply that $P$ is Markov thanks to Lemma~\label{markovfromgluing}.
\end{proof}

As a consequence of Lemma \ref{lemmarkov} and the Markovianity of the reversible Wiener measure we deduce

\begin{coro}\label{snegalen}
Let $(\mu_0, \cdots, \mu_N)\in (\PP_2(\R^d))^{N+1}$, one has
\[\begin{split}
\SS^N(\mu_0, \cdots, \mu_N) &=\sum_{i=0}^{N-1} S_{\frac{T}{N}} (\mu_i, \mu_{i+1}) -\sum_{i=1}^{N-1} \Ent(\mu_i)\\
&=\E^N(\mu_0, \cdots, \mu_N)+\Ent(\mu_0).
\end{split}\]
\end{coro}

\begin{proof}
Let us define the Markov probability measure $\tilR^N$ by \pref{tilRNw}. Then we have
\[\SS^N(\mu_0, \cdots, \mu_N)=\inf_{\gamma^N\in \Pi(\mu_0, \cdots, \mu_N)} H(\gamma^N \vert \tilR^N)-H(\mu_0\vert \tilR_0)+ \Ent(\mu_0)\]
The Markovianity of $\tilR^N$ and Lemma \ref{lemmarkov} then give
\[\SS^N(\mu_0, \cdots, \mu_N)=  \Ent(\mu_0) +
\inf_{\gamma^N\in \Pi(\mu_0, \cdots, \mu_N)}  \Big\{ \sum_{i=0}^{N-1} (H(\gamma_{i, i+1}^N \vert \tilR_{i, i+1}^N)-H(\mu_i \vert \tilR^N_i)\Big\}\]
Now we observe that 
\[\begin{split}
R^N_{i, i+1}(\mbox{d} x_i \mbox{d} x_{i+1})= \mbox{d} x_i R^N_{i,i+1}(\mbox{d} x_{i+1} \vert x_i), \\
\tilR^N_{i, i+1}(\mbox{d} x_i \mbox{d} x_{i+1})= \tilR^N_{i}(x_i) \mbox{d} x_i R^N_{i,i+1}(\mbox{d} x_{i+1} \vert x_i)
\end{split}\]
with $\tilR^N_{i}$ being Gaussian so that  $\log(\tilR_i^N)\in L^1(\mu_i)$, we deduce
\[\begin{split}
H(\gamma_{i, i+1}^N \vert \tilR_{i, i+1}^N)-H(\mu_i \vert \tilR_i)&=H(\gamma_{i, i+1}^N \vert R_{i, i+1}^N)- \int_{\R^d} \log(\tilR_i^N) \mu_i\\
&-\Ent(\mu_i)+ \int_{\R^d} \log(\tilR_i^N) \mu_i\\
&=H(\gamma_{i, i+1}^N \vert R_{i, i+1}^N)-\Ent(\mu_i).
\end{split}\]
Hence 
\[\begin{split}
\SS^N(\mu_0, \cdots, \mu_N)&=  \Ent(\mu_0)  +\sum_{i=0}^{N-1} \Big( \inf_{\gamma\in \Pi(\mu_i, \mu_{i+1})} H(\gamma \vert R_{i, i+1}^N) -\Ent(\mu_i) \Big)\\
&=\sum_{i=0}^{N-1} S_{\frac{T}{N}} (\mu_i, \mu_{i+1}) -\sum_{i=1}^{N-1} \Ent(\mu_i)
\end{split}\]
 the second identity in the Lemma  then directly follows from the first one and \pref{lienfps}. 
\end{proof}

\section{$\Gamma$-convergence and equivalence}\label{sec-gamma}

\subsection{$\Gamma$-convergence}

We state now some lemmas which  will be useful for the proof of
the $\Gamma$-convergence result, stated in Theorem \ref{gammaconvfp}.

\begin{lem}\label{lem:dyna} Let $\{ \rho_t \}_{0 \leq t \leq T} $ be a curve of probability measures such that $\partial_t \rho_t + \nabla \cdot ( v_t \rho_t) = \frac {1}2 \Delta \rho_t$ with $\iint_0^T |v_t|^2 \, d \rho_t =E < \infty$. Let us suppose also that $\int_{\R^d} |x|^2 \, d \rho_0 = c_0 < \infty$ is finite. Then we have:

\begin{itemize}

\item[(i)] $\int |x|^2 \, d \rho_t \leq C(T, c_0, E)$ for all $0\leq t \leq T$.

\item[(ii)] $\Ent ( \rho_t) \in L^{\infty} ( \tau, T)$ for all $\tau >0$.

\item[(iii)] $ I(\rho_t) \in L^1(\tau, T)$ and $\Ent (\rho_t) \in W^{1,1}(\tau, T)$ for all $\tau>0$.

\item[(iv)] $ \iint_{\tau}^T | v_t|^2 \, d \rho_t = \iint_{\tau}^T |w_t|^2 \, d\rho_t + \frac{ 1}4 \int_{\tau}^T I(\rho_t) \, dt +  (\Ent ( \rho_T) - \Ent(\rho_{\tau}))$, where $w_t = v_t - \frac 1 2 \nabla \log (\rho_t)$.
\item[(v)]  $ \W_2^2 (\rho_{\tau} ,  \rho_T) \leq (T-\tau) ( E +  (\Ent( \rho_{\tau} ) - \Ent (\rho_T)) ) $. In particular we have that $\rho_t  \in H^1([\tau, T ]; ( \mathcal{P}_2(\R^d) ,\W_2) )$. 
\end{itemize}

Moreover if we add the assumption $\Ent(\rho_0) < \infty$ we can take also $\tau=0$ in (ii), (iii), (iv) and (v).

\end{lem}

\begin{proof} Let us take $\delta >0$ and let us consider $\rho^{\delta}_t= \eta_{\delta} * \rho_t$ a convolution with a $C^{\infty}$ (and everywhere strictly positive) kernel $\eta_{\delta}=\eta(x/\delta)/\delta^d$ with a finite second moment. Denoting $m^{\delta}_t= \eta_{\delta} * (v_t \rho_t) $ we have that $\partial_t \rho_t^{\delta} + \nabla \cdot ( m^{\delta}_t) = \frac 1 2 \Delta \rho^{\delta}_t$. In particular setting $v_{t}^\delta = m^{\delta}_t /\rho^{\delta}_t$, we have, by convexity of the function $(v,s) \mapsto \frac {|v|^2}s$

$$ \iint_0^T | v^{\delta}_t|^2 \, d \rho_t^{\delta} = \iint_0^T \frac{| (v_t\rho_t) *\eta_{\delta}|^2}{\rho_t*\eta_{\delta}} \, d x  \leq   \iint_0^T \frac{| v_t \rho_t |^2}{\rho_t} \, d x = E.$$

Moreover we can compute, letting $c_t^{\delta} = \int |x|^2 \, d \rho^{\delta}_t$:
$$ \frac{ d}{dt} c_t^{\delta} = 2 \int \langle x , v_t^{\delta} \rangle \, d \rho^{\delta}_t +  2d \leq  c_t^{\delta} + \int |v_t^{\delta}|^2\, d \rho_t^{\delta} + 2d $$ 
$$ \frac{ d}{dt} ( c_t^{\delta} e^{-t} ) \leq  e^{-t} \left(\int |v_t^{\delta}|^2\, d \rho_t^{\delta} + 2d \right)  \leq  \int |v_t^{\delta}|^2\, d \rho_t^{\delta} + 2d. $$ 
So we get 
\begin{equation}\label{moment-bound}
c_t^{\delta} \leq  e^{t} ( c^{\delta}_0 + E+  2d t ) \leq  e^{T}( c_0+  \delta^2 \int \vert x\vert^2 \eta+ 2 \delta (\int x \mbox{d} \rho_0)(\int x \eta) +E+ 2d T)
\end{equation}
now we can let $\delta$ to $0$ to obtain $(i)$. Moreover, we can obtain a bound on the entropy. Indeed,  $e^{-\pi |x|^2}$ is the density of a probability measure and so  for any $\rho$ bounded density of a probability measure with finite second moment we have:

$$0 \leq  H ( \rho | e^{-\pi \vert x\vert^2} )   = \pi \int |x|^2 \, d \rho + \int \rho \log (\rho ) \, d x$$
\begin{equation}\label{entrminus-bound}
 \int \rho \log (\rho) \, dx \geq -\pi \int |x|^2 \, d \rho.
 \end{equation}

Taking the function $f(t) = e^{- \frac{2}{d} \Ent ( \rho_t^{\delta})}$, we have:

\begin{align*}
\frac{ d}{dt} f(t) &= - \frac {2f(t)} d  \left( \int \nabla \rho_t^{\delta} \cdot v_t^{\delta} dx - \frac 1 2 \int \frac{|\nabla  \rho_t^{\delta}|^2}{\rho_t^{\delta}} \, dx \right )  \\
 & \geq  - \frac {2f(t)} d \left(  \int |v^{\delta}_t|^2\, d \rho_t^{\delta}  + \frac 1 4  \int   \frac{|\nabla  \rho_t^{\delta}|^2}{\rho_t^{\delta} }  \, dx - \frac 1 2 \int  \frac{|\nabla  \rho_t^{\delta}|^2}{\rho_t^{\delta}} \, dx \right ) \\
 & =  - \frac {2f(t)} d \left(  \int |v^{\delta}_t|^2\, d \rho_t^{\delta}  -   \int     |\nabla \sqrt{  \rho_t^{\delta}} |^2     \, dx \right )\\
  & \geq \pi e - \frac {2f(t)}{d} \int |v_t^{\delta}|^2\, d \rho_t^{\delta},
\end{align*}

where in the last passage we used the Log-Sobolev inequality\footnote{More precisely, the optimal euclidean Logarithmic Sobolev inequality  in the form 
\[\Ent(\rho)\leq \frac{d}{2} \log \Big(\frac{2}{\pi e d} \int_{\R^d} \vert \nabla \sqrt{\rho}\vert^2 \Big), \]
see Weissler \cite{Weissler}.}. So we can conclude $f(t) \geq  e \pi  t  e^{-\frac 2{d}\iint_0^t |v_s^{\delta}|^2 \, d \rho_s^{\delta} } $, that is 
$$\Ent ( \rho_t^{\delta}) \leq - \frac d 2 \log \left(  \pi e t  \right) +  \iint_0^t |v_s^{\delta}|^2 \, d \rho_s^{\delta}. $$
This proves $(ii)$. In order to establish $(iii)$, we first notice that 

\begin{equation}\label{eqn:entropyw11} \Ent (\rho_r^{\delta}) -\Ent(\rho_s^{\delta}) = \iint_s^r v_t^{\delta} \cdot \nabla \rho_t^{\delta} \, dxdt - \frac 1 2 \iint_s^r \frac {|\nabla \rho_t^{\delta}|^2}{\rho_t^{\delta}} \, dxdt. \end{equation}
In particular taking  $s=\tau$ and $r=T$ and using $v_t^{\delta} \cdot \nabla \rho_t^{\delta} \leq  |v_t^{\delta}|^2 \rho_t^{\delta} + \frac 1 4 \frac {|\nabla \rho_t^{\delta}|^2}{\rho_t^{\delta}}$ we get, thanks to the $L^{\infty}$ bound on the entropy:

\begin{equation}\label{eqn:Ienergy} \frac 1 4 \iint_\tau^T \frac {|\nabla \rho_t^{\delta}|^2}{\rho_t^{\delta}} \, dxdt \leq E + \Ent(\rho_\tau^{\delta}) - \Ent(\rho_T^{\delta}) \leq C(\tau, T, E, c_0),
\end{equation}
which proves that $I(\rho_t)$ is in $L^1(\tau, T)$. Equation \eqref{eqn:entropyw11} also yields

\begin{equation} \label{eqn:entw11}
| \Ent(\rho_t^{\delta}) - \Ent(\rho_s^{\delta})| \leq \int_s^t \left(2 I(\rho_t^{\delta}) + \frac 1{4} \int_{\R^d} \vert v_t^{\delta}\vert^2 \rho_t^{\delta}\right) \, dt,
\end{equation}
by convexity of $I$ we have $0\le I(\rho^\delta_t)\le I(\rho_t)\in L^{1} (\tau, T)$ and a similar convexity argument gives an integrable bound for the kinetic energy as well, hence
$\Ent(\rho_t) \in W^{1,1} (\tau, T)$ for every $\tau >0$. Statements $(iv)$ and $(v)$ easily follow. Finally, whenever $\Ent(\rho_0) < \infty$ from equation \eqref{eqn:Ienergy} we get immediately $I(\rho_t) \in L^1(0,T)$ and then from  \eqref{eqn:entw11} we have $\Ent(\rho_t) \in W^{1,1} (0, T)$. 

\end{proof}

%
%

\begin{lem}\label{lem:reparametrization} Let $\mu$ and $\nu$ be two fixed probability measures such that $E:= 2 \FP_T(\mu, \nu) < \infty$. Let $c_0 = \int |x|^2 \, d \mu <\infty$ and let us fix $T/2 \leq T'\leq 2T$. Then we have
\begin{itemize}
\item[(i)] $\Ent(\nu) \leq \Ent(\mu) + E$;
\item[(ii)] $\Ent(\nu) \geq - \pi e^{T}( c_0 +E+ 2d T)$;
\item[(iii)] $ 2 \FP_{T'} (\mu, \nu) \leq  5 E +  (\Ent(\mu) - \Ent(\nu)) $.
\end{itemize}
\end{lem}

\begin{proof}

Inequality $(i)$ directly follows from Lemma \ref{lem:dyna}-$(iv)$; $(ii)$  is implied directly by \pref{moment-bound}-\pref{entrminus-bound} in the proof of lemma \ref{lem:dyna}. As for the last estimate let us consider an almost minimiser $(\rho_t,v_t)$ for the problem $ \FP_T(\mu, \nu) $; in particular we have $\partial \rho_t + \nabla ( v_t \rho_t) = \frac 1 2 \Delta \rho_t$ and $\iint_0^T v_t^2 \,d \rho_t \leq E+\delta$. Let us take $(\tilde{ \rho}_t, \tilde{v}_t)  = (\rho_{Tt/T'}, v_{Tt/T'})$; this curve is such that $\tilde{\rho}_{T'}=\nu$ and $\tilde{\rho}_0=\mu$ and moreover $ \partial_t \tilde{\rho}_t + \frac{T}{T'}\nabla ( \tilde{v}_t \tilde{\rho}_t ) = \frac {T}{2T'} \Delta \tilde{\rho}_t$ and in particular we have

$$ \partial_t \tilde{\rho}_t + \nabla \cdot \left( \Bigl( \frac{T}{T'}\tilde{v}_t  + \frac{1}2\bigl(1-T/T'\bigl)\frac {\nabla \tilde{\rho}_t}{\tilde{\rho}_t} \Bigr)\tilde{\rho}_t \right) = \frac {1}{2} \Delta \tilde{\rho}_t.$$

We can thus estimate

\begin{align*}
 2 \FP_{T'} (\mu, \nu)  &\leq \iint_0^{T'} \Bigl \vert \frac{T}{T'}\tilde{v}_t  + \frac{1}2\bigl(1-T/T'\bigl)\frac {\nabla \tilde{\rho}_t}{\tilde{\rho}_t} \Bigr \vert^2 \, d\tilde{\rho}_t \, dt\\
& \leq  2 \iint_0^{T'} \frac {T^2}{T'^2} \vert \tilde{v}_t\vert^2 \, d\tilde{\rho}_t + 2 \Big(1- \frac{ T}{T'}\Big)^2  \iint_0^{T'} \frac{ 1}{4}  \frac {|\nabla \tilde {\rho}_t|^2}{\tilde{\rho}_t} \,dx \, dt \\
& \leq  \frac{ 2T}{T'} (E +  \delta) + \frac{2  T'}{T} \Big(1- \frac{ T}{T'}\Big)^2   \int_0^T \frac{1}{4} I(\rho_t) \, dt\\
& \leq \frac{ 2T}{T'} (E +  \delta)+ \frac{2  T'}{T} \Big(1- \frac{ T}{T'}\Big)^2 (E+\delta+ \Ent(\mu) - \Ent(\nu))\\
& \leq 5 (E+\delta) +    (\Ent(\mu) - \Ent(\nu)).
\end{align*}

In the last two inequalities we used the estimates on $\int I(\rho_t)$ given by lemma \ref{lem:dyna}-$(iv)$ and then the fact that $T'/T \in [1/2, 2]$ which is easily seen to imply that $\frac{2  T'}{T} (1- \frac{ T}{T'})^2 \le 1$. Letting $\delta \to 0 $ we  thus obtain $(iii)$.

\end{proof}

We now fix the initial condition $\rho_0$ such that
\begin{equation}\label{condrhoz}
\rho_0\in \PP_2(\R^d), \; \Ent(\rho_0)<+\infty,
\end{equation}
and define 
\begin{equation}
\A_{\rho_0}:= \{\mu \in C([0,T], (\PP_2(\R^d), \W_2)) \; : \; \mu_0=\rho_0\}
\end{equation}
and its discretization
\begin{equation}
\A^N_{\rho_0}:= \{\mu^N:= (\mu_0^N, \cdots, \mu_N^N) \in \PP_2(\R^d)^{N+1}, \;   \mu_0=\rho_0\}.
\end{equation}
Given $\mu^N \in \A^N_{\rho_0}$, it will be convenient to extend it in a piecewise constant way as
\begin{equation}
\tilmu_t^N:=\mu_k^N,\; \; t\in \Big(\frac{(k-1)T}{N}, \frac{kT}{N}\Big], \; k=1, \cdots, N, \; \tilmu_0:=\mu_0^N=\rho_0.
\end{equation}
By construction, we  have
\[\CC^N(\mu^N)=\int_0^{T-T/N} F(\tilmu^N_s) ds + G(\tilmu_T^N).\]
We will say that a sequence $\mu^N\in \A^N_{\rho_0}$ converges to a curve of measures $\mu\in \A_{\rho_0}$ whenever
\begin{equation}\label{defconver}
 \lim_{N\to \infty} \sup_{t\in[0,T]} \W_2(\tilmu^N_t, \mu_t)=0.
 \end{equation}

For further use, let us point out that if $\E^N(\mu^N)$ is bounded, it follows from Lemmas \ref{lem:dyna}-\ref{lem:reparametrization} and a refined version of Ascoli-Arzel\`a's theorem (see \cite{AGS} chapter 3) that (up to a subsequence), $\mu^N$ may be assumed to converge (in the sense of \pref{defconver}) to some $\mu\in \A_{\rho_0}$.  Thanks to  Lemmas \ref{lem:dyna}, we have uniform (in $t$ and $N$) bounds on second moments and entropy of $\tilmu_t^N$ hence we may also assume that $\tilmu^N$ converges to $\mu$ weakly in $L^1((0,T)\times \R^d)$ and also that $\tilmu_T^N$ converges weakly in $L^1(\R^d)$ to $\mu_T$.


\begin{thm}\label{gammaconvfp}
The sequence of functionals $\E^N+\CC^N$ : $\A^N_{\rho_0}\to \R_+\cup \{+\infty\}$ $\Gamma$-converges to the functional $\E+\CC$: $\A_{\rho_0} \to  \R_+\cup \{+\infty\}$ as $N\to \infty$.  

\end{thm}

\begin{proof}

Let us start with the $\Gamma$-liminf inequality. Let $\mu^N\in \A^N_{\rho_0}$ converge  (in the sense of \pref{defconver}) to $\mu\in \A_{\rho_0}$ and assume that $\E^N(\mu^N)+\CC^N(\mu^N)$ is bounded (hence so  is $\E^N(\mu^N)$). As observed above, one may assume that $\tilmu^N$ converges to $\mu$ weakly in $L^1(0,T)\times \R^d)$ and also that $\tilmu_T^N$ converges weakly in $L^1(\times \R^d)$ to $\mu_T$. In the terminal  term $G(\tilmu_T^N)=\int_{\R^d} L_G(\tilmu_T^N(x)) dx+ G_2(\tilmu_T^N)$, we have the sum of a convex nonnegative integral term, which is then weakly l.s.c. in $L^1(\R^d)$, and a term, $G_2(\tilmu_T^N)$, which is continuous for $\W_2$, so that 
\begin{equation}\label{lscboundary}
\liminf_N G(\tilmu_T^N)\ge G(\mu_T).
\end{equation}
Fix $\eps>0$ and let $N>T \eps^{-1}$, by Fatou's Lemma, the non-negativity of $F=F_1+F_2$ and invoking the same weak $L^1$ lower-semi continuity argument as above (but with respect to both variables $t$ and $x$), we get
\[\liminf_N \int_0^{T-T/N} \Big(  \int_{\R^d} L_F(\tilmu_s^N) dx+ F_2(\tilmu_s^N) \Big)ds \ge \int_0^{T-\eps} F(\mu_s) ds\]
letting $\eps\to 0$ and recalling \pref{lscboundary}, we  deduce
\begin{equation}\label{lsccn}
\liminf_{N} \CC^N(\mu^N) \ge \CC(\mu). 
\end{equation}
For the energy term, we consider $m_t^N :=\tilmu_t^N v_t^N$, an almost minimizer for $\E^N(\mu^N)$ i.e. $\partial_t \tilmu_t^N-\frac{1}{2} \Delta  \tilmu_t^N+ \dive(m_t^N)=0$ and
\begin{equation}
\E^N(\mu^N) \ge \frac{1}{2} \int_0^T \int_{\R^d} \frac{\vert m_t^N\vert^2}{\tilmu_t^N} -\frac{1}{N}.
\end{equation}
By Young's inequality and the bound on $\E^N(\mu^N)$, one deduces a bound on the total variation (in $t$ and $x$) of $m^N$ from which (up to a further extraction) one may assume that for every $R>0$, $m^N$ converges weakly $*$ to some vector-valued measure $m$ on $[0,T]\times B_R$, where $m$ satisfies $\partial_t \mu_t-\frac{1}{2} \Delta  \mu_t+ \dive(m_t)=0$. By a well known lower semi-continuity property of convex functionals of measures (see Theorem 2.34 in \cite{AFP}) which applies to the Benamou-Brenier functional, we have
\[\liminf_{N} \E^N(\mu^N) \ge \frac{1}{2} \int_0^T \int_{B_R} \frac{\vert m \vert^2}{\mu}.\]
Letting $R\to \infty$ we deduce that $m=\mu v$ and $\int_0^T \int_{\R^d} \vert v_t\vert^2 d \mu_t dt <+\infty$ so that 
\[\liminf_{N} \E^N(\mu^N) \ge  \frac{1}{2} \int_0^T \int_{\R^d} \vert v_t\vert^2 d \mu_t dt \ge \E(\mu).\]
Combining it with \pref{lsccn} we obtain the desired $\Gamma$-liminf inequality.

\smallskip

For the $\Gamma$-limsup inequality we will do a smoothing construction in order to deal with the nonlinear term $F$. For $\delta \in [1/2, 3/2]$, let us consider the time reparametrization $\phi_{\delta,N}$  defined by $\phi_{\delta,N}(0)=0$ and
$$ \phi_{\delta,N}' (t)= \begin{cases} \delta \qquad & \text{if } 0\le  t \leq T/N \\ 1 \qquad & \text{if }  T/N \leq  t \leq T- T/N \\ 2-\delta \qquad & \text{if } T-T/N  \le t  \le T \end{cases}.$$
Notice in particular, that, if $\frac 12 \leq \delta \leq \frac 32$, then 
\begin{equation}\label{eqn:propphi}
 \frac 12 (t-s) \leq \phi_{\delta,N}(t)- \phi_{\delta,N}(s) \leq 2 (t-s), \;    \vert \phi_{\delta,N}(t)-t\vert \le \frac{T}{2N}, \;  \forall 0\leq s \leq  t \leq T. 
\end{equation}
Let $\mu\in \A_{\rho_0}$ be such that $\E(\mu)+\CC(\mu)<+\infty$ (otherwise there is nothing to prove). We then define 
\[\tilde{\rho}^N_t =\int_{1/2}^{3/2} \mu_{\phi_{\delta,N}(t)} \, d \delta.\]
In particular it is worthwhile to notice that
\begin{equation}\label{eqn:linear}\tilde{ \rho}^N_t =  \frac N {T}\int_{t-\frac T{2N}}^{t+\frac T{2N}} \mu_s \, ds \qquad  \text{if } T/N \leq t \leq T-T/N.\end{equation}

Our recovery sequence will be then $\mu_k^N = \tilde{ \rho}^N_{kT/N}$, $k=0, \cdots, N$. Notice  that $\mu^N$ converges to $\mu$ thanks to \pref{eqn:propphi} and the fact that the curve $t\in [0,T]\mapsto \mu_t$ is $C^{0,1/2}$ with respect to $\W_2$ as a consequence of Lemma \ref{lem:dyna}. By construction, $\mu_N^N=\mu_T$, so the convergence of the terminal term $G$ is trivial. The convergence of the discretization of $F_2$ then follows directly from its continuity in the Wasserstein metric and Lebesgue's dominated convergence theorem. For the  integral term $F_1$, we use equation \eqref{eqn:linear}, the convexity and the non-negativity of $F_1$ to conclude that 
\[\begin{split}
\frac T N \sum_{k=1}^{N-1} F_1(\mu_k^N) \leq \frac T N \sum_{k=1}^{N -1}  \frac N {T}\int_{\frac{kT}N-\frac T{2N}}^{\frac{kT}N+\frac T{2N}} F_1(\mu_s) \, ds  \\
= \int_{  \frac{T}{2N}}^{T- \frac{T}{2N}} F_1(\mu_s)\ ds
\leq  \int_0^T F_1(\mu_s)\ ds 
\end{split}\]
hence 
\begin{equation}\label{glsc}
\limsup_{N \to \infty} \CC^N(\mu^N)\le \CC(\mu).
\end{equation}
For the energy part, we first use the convexity of $\FP_{\frac{T}{N}}$ and the representation of $\tilde{\rho}^N_t$ to get that for $1\le k \le N-2$

$$\FP_{\frac{T}{N}}  (\mu_k^N , \mu_{k+1}^N) \leq \frac N{T} \int_{kT/N-T/2N}^{kT/N+T/2N}  \FP_{\frac{T}{N}}(\mu_{s} , \mu_{s+ T/N}) \, ds.$$

Now we take $v$ an optimal vector field for $\E(\mu)$ (existence can be obtained by classical arguments, indeed, taking $m=\mu v$ as new unknown, this is a convex minimization problem with a linear constraint). By construction, we then have 
$ \FP_{\frac{T}{N}}(\mu_{s} , \mu_{s+ T/N}) \leq \frac{1}{2} \int_s^{s+T/N} \int_{\R^d} \vert v_t\vert^2 d\mu_t dt$. So, summing up over $k$ we get
\[\begin{split}
\sum_{k=1}^{N-2} \FP_{\frac{T}{N}}  (\mu_k^N , \mu_{k+1}^N) \leq \frac N{T} \sum_{k=1}^{N-2}  \int_{kT/N-T/2N}^{kT/N+T/2N}  \int_s^{s+T/N}  \frac{1}{2}  \int_{\R^d} \vert v_t\vert^2 d\mu_t dt ds\\
\le \int_{T/2N}^{T-T/2N}   \frac{1}{2}  \int_{\R^d} \vert v_t\vert^2 d\mu_t dt \le \E(\mu)
\end{split}\]


For the last pieces, we use  again the convexity of $\FP_{\frac{T}{N}}$, the definition of $\mu_1^N=\int_{1/2}^{3/2} \mu_{st/N} ds$ and  Lemma \ref{lem:reparametrization} $(iii)$ to get

\begin{align*}
\FP_{\frac{T}{N}}  (\mu_0^N , \mu_{1}^N) &= \FP_{\frac{T}{N}}  (\rho_0 , \mu_{1}^N) \\
&\leq \int_{1/2}^{3/2}    \FP_{\frac{T}{N}} (\rho_0, \mu_{sT/N} ) \, ds \\ 
&\leq 5 \int_{1/2}^{3/2}  \FP_{\frac{sT}{N}} (\rho_0, \mu_{sT/N} ) \, ds  +  \frac{1}{2} \int_{1/2}^{3/2} \Ent (\rho_0) - \Ent(\mu_{sT/n}) \, ds\\
& \leq 5 \int_0^{\frac{3T}{2N}} \int_{\R^d} |v_t|^2 \, d \rho_t dt + \frac{1}{2} \int_{1/2}^{3/2} \Ent (\rho_0) - \Ent(\mu_{sT/N}) \, ds
\end{align*}

But now the first term is converging to zero as $N \to \infty$ and the second term as well thanks to the uniform continuity of $t \mapsto \Ent (\mu_t)$ (proven in Lemma \ref{lem:dyna} $(iii)$). In a similar way, we can also prove  that $ \FP_{\frac{T}{N}} (\mu_{N-1}^N , \mu_{N}^N) \to 0$ thus completing the proof that
\[\limsup_{N\to \infty} \E^N(\mu^N)\le \E(\mu)\]
which, together with \pref{glsc}, gives the desired $\Gamma$-limsup inequality for the recovery sequence $\mu^N$. 

\end{proof}

\begin{rem}\label{moregeneralFG}
For the sake of simplicity, we have taken a continuous local part $L_F$ but is clear from the proof of Theorem \ref{gammaconvfp} that our results extend to the case where there is an additional  pointwise closed and convex state constraint on the density, for instance a hard congestion constraint such as $\rho \le M$ where $M$ is a prescribed saturation threshold (a case we will actually consider in our numerical simulations in section \ref{sec-algo}). It also goes without saying that allowing $F$ and $G$ to depend on $x$ does not create extra difficulties neither. 
\end{rem}

In particular, when $F=G=0$, we  have

\begin{coro}\label{corogamma}
The sequence of functionals $\E^N$: $\A^N_{\rho_0}\to \R_+\cup \{+\infty\}$ $\Gamma$-converges to the functional $\E$: $\A_{\rho_0} \to  \R_+\cup \{+\infty\}$ as $N\to \infty$. 

\end{coro}

\subsection{Equivalence}

Our aim now is to show that thanks to Theorem \ref{gammaconvfp}, we have a continuous in time counterpart of the formulas of Corollary \ref{snegalen}. This will enable us to deduce that \pref{minimfgsynt} can be equivalently  rewritten as an entropy minimization problem. We first need a technical result:

\begin{lem}\label{compactQ}
Let $\rho_0$ satisfy \pref{condrhoz}, if $(Q^N)_N$ is a family in $\PP(\Omega)$ such that $Q_0^N=\rho_0$ and $H(Q^N\vert R)$ is bounded then $Q^N$ is tight. Moreover, if there exists $\mu\in A_{\rho_0}$ such that $Q^N_{kT/N}=\mu_{kT/N}$ for every $k$ and $Q^N$ converges narrowly to $Q$ then $Q_t=\mu_t$ for every $t\in [0,T]$. 
\end{lem}

\begin{proof}
Let us first define the (equivalent to $R$) probability measure $\tilR$ on $\Omega$ by $\tilR(\mbox{d}\omega):=e^{-\pi \vert \omega(0)\vert^2} R(\mbox{d}\omega)$. We then observe that 
\[H(Q^N\vert R)=H(Q^N \vert \tilR)-\pi \int_{\R^d} \vert x\vert^2 \mbox{d} \rho_0.\]
Since $Q^N$ is absolutely continuous with respect to $\tilR$, it can be identified to its Radon-Nikodym derivative $\tilQ^N$ with respect to $\tilR$,
we thus deduce from the boundedness of $H(Q^N\vert R)$  that $\tilQ^N$ is uniformly integrable hence admits a subsequence which converges weakly in $L^1(\tilR)$, the tightness claim directly follows. Let us now prove the second statement,


Since $\Omega$ is Polish and $Q^N$ is tight, by Prokhorov Theorem, for every $\eps>0$ there exists a compact subset $K_\eps$ of $\Omega$ such that $Q^N(\Omega\setminus K_\eps)\leq \eps$ for every $N$. Now let $t\in [0,T]$ and $t_N=k_NT/N$ be such that $\vert t- t_N\vert \le  T/N$. For a given $\varphi \in C_b^0(\R^d)$ we then have
\[\int_{\R^d} \varphi \mbox{ d} (Q_t^N-\mu_t) =\int_{\R^d} \varphi \mbox{ d} ((Q_t^N-Q^N_{t_N})-(\mu_t-\mu_{t_N}))\]
the second term tends to $0$ by continuity of the curve $t\mapsto \mu_t$ and we can decompose the first term  as
\[\int_{K_\eps}( \varphi(\omega(t))-\varphi(\omega(t_N)) ) \mbox{d} Q^N(\omega)+ \int_{\Omega \setminus K_\eps} ( \varphi(\omega(t))-\varphi(\omega(t_N))) \mbox{d} Q^N(\omega)\]
By Ascoli-Arz\'ela's Theorem, $K_\eps$ is uniformly equicontinuous so that the first term tends to $0$ as $N\to +\infty$ whereas the second term is less in absolute value than $2\eps \Vert \varphi\Vert_{\infty}$. Since $Q^N_t$ converges narrowly to $Q_t$ this shows that $Q_t=\mu_t$. 
\end{proof}

As a consequence, we recover,  the representation result of Dawson and G\"{a}rtner \cite{DawsonGartner} (also see  \cite{Follmer}):

\begin{coro}
Let $\rho_0$ satisfy \pref{condrhoz}, then we have
\begin{equation}\label{segale}
\SS(\mu)=\E(\mu)+\Ent(\rho_0), \; \forall \mu \in \A_{\rho_0}.
\end{equation}

\end{coro}

\begin{proof}
Let $\mu \in \A_{\rho_0}$ be such that $\SS(\mu)<+\infty$, by definition of $\SS^N$ and $\SS$ and using Corollary \ref{snegalen}, we have
\[\SS(\mu)\ge \SS^N(\mu_0, \mu_{T/N}, \cdots, \mu_T) =\E^N(\mu_0, \mu_{T/N}, \cdots, \mu_T)+\Ent(\rho_0).\]
Taking the liminf and using Theorem \ref{gammaconvfp}, we get $\SS(\mu)\ge \E(\mu)+ \Ent(\rho_0)$.  Now take  $\mu \in \A_{\rho_0}$ such that $\E(\mu)<+\infty$, we then have
\[\E(\mu)\ge \E^N(\mu_0, \mu_{T/N}, \cdots, \mu_T) =\SS^N(\mu_0, \mu_{T/N}, \cdots, \mu_T)-\Ent(\rho_0)\]
let then $Q^N \in \PP(\Omega)$ be such that $Q_{kT/N}=\mu_{kT/N}$ for  $k=0, \cdots, N$ and 
\[ \SS^N(\mu_0, \mu_{T/N}, \cdots, \mu_T) \ge H(Q^N \vert R)-\frac{1}{N}\]
passing to a  subsequence if necessary, thanks to Lemma \ref{compactQ}, we may assume that $Q^N$ converges narrowly to some $Q\in \PP(\Omega)$ such that $Q_t=\mu_t$ for every $t\in [0,T]$ so that, thanks to the weak lower semi-continuity of the entropy
\[\E(\mu)+\Ent(\rho_0)\ge \liminf_N H(Q^N \vert R) \ge H(Q\vert R) \ge \SS(\mu).\]

\end{proof}

\section{Algorithm and Numerical Results}\label{sec-algo}

\subsection{Multi-Marginal Sinkhorn} 

We now introduce  a numerical scheme to solve the discretized in time problem  \eqref{entminimfgsyntn}. The scheme is based on  a variant of the celebrated Sinkhorn algorithm \cite{Sinkhorn67}, already used to solve many variational problem related to optimal transport (see for instance \cite{Cut,galichon2015cupid,benamou2017generalized,benamou2016numerical,Ben,blanchet2016computation,chizat2016scaling,thesislulu,peyre2015entropic}). 
Note that the iterative method proposed in \cite{gueant} for MFGs with quadratic hamiltonians can be reinterpreted in this framework. \\

 We recall that \eqref{entminimfgsyntn} reads as:
\[ \inf\{ \SS^N(\mu_0, \cdots, \mu_N)+  \frac{T}{N}\sum_{k=1}^{N-1}F(\mu_k)+G(\mu_N),  \; \mu_i \in \PP_2(\R^d), \;  \mu_0=\rho_0\}\]
where $\SS^N$ is itself defined by   \eqref{discrshrosta} which is an entropy minimization with multi-marginal constraints.
Denoting    $\pi^k:(\R^d)^{N+1}\rightarrow(\R^d)$  the $k-$th canonical projection we can obviously rewrite \eqref{entminimfgsyntn}  as  an optimization problem over plans $\gamma^N$ only:
\begin{equation}
\label{primalEnt}
 \inf\left\{ H(\gamma^N\vert R^N) +i_{\rho_0} (\pi_\#^0\gamma^N)+\frac{T}{N}\sum_{k=1}^{N-1}F(\pi^k_{\#}\gamma^N)+G(\pi^N_{\#}\gamma^N)\;:\;\gamma^N\in\PP((\R^d)^{N+1})\right\} ,
 \end{equation} 
 where 
 \[i_{\rho_0}(\rho)= \begin{cases} 0 \mbox{ if $\rho=\rho_0$}\\ +\infty \mbox{ otherwise} \end{cases}\]
  is the indicator function in the convex analysis sense and is used to enforce the initial  condition.
 
\smallskip
We recall that $R^N$ is defined as  
\[ R^N:=R_{0,\frac{T}{N},\cdots,T}  \]
Since $R$ is the reversible Wiener measure, $R^N$ can be decomposed by using the heat kernel \eqref{heatker} as
\begin{equation}
\label{decomposed}
 R^N( \mbox{d} x_0,\cdots, \mbox{d} x_N):=\left(\prod_{k=1}^N P_{\frac{T}{N}}(x_k-x_{k-1})\right)\mbox{d}x_0\cdots\mbox{d}x_N .
 \end{equation}

We also need to discretize in space (for example we use $M$ grid-points to discretize $\R^d$), then $\gamma^N$ and  $R^N$ become tensors in  $\R^{MN}$ 
 (see remark \ref{implementation} for a further decomposition of $R^N$ ). Under the assumption that $F$ and $G$ are convex,  \pref{primalEnt} is now a finite-dimensional strictly convex minimization\footnote{We will also allow $F$ to contain a nonconvex but regular term given by a convolution which we will treat in a semi-implict way as explained in paragraph \ref{nonlocal}.}.  In order to simplify the presentation, 
 we will keep the continuous in space notation. Integrals must therefore be understood as  finite sums and $x_0,..x_N$ as $M$ vectors. 

\smallskip

One can now generalise the algorithm formalized in  \cite[theorem 3.2]{chizat2016scaling}  (the case $N=1$ or two marginals) as follows. We first state without proof a classical duality result: 
\begin{prop}
The dual problem of \eqref{primalEnt} is
\begin{equation}
\label{dualEnt}
\sup_{(u_0,\cdots,u_N)}-\tilde F^\star(-u_0)-\frac{T}{N}\sum_{k=1}^{N-1}F^\star(-u_k)-G^\star(-u_N)-\int \left(\exp(\oplus_{k=0}^{N}u_k)-1\right)R^N
\end{equation}
where $\oplus_{k=0}^{N}u_k: (x_0,\cdots,x_N)\mapsto u_0(x_0)+\cdots+u_{N}(x_N)$. Strong duality holds in the sense that the minimum in \pref{primalEnt} coincides with the maximum in \pref{dualEnt}.
\end{prop}
Denoting by $u_k^\star$ and $\gamma^\star$ the optimal solutions to \eqref{dualEnt} and \eqref{primalEnt} respectively, it follows that the unique solution to \eqref{primalEnt} has the form
\begin{equation}
\gamma^\star(x_0,\cdots,x_N):=\left (\otimes_{k=0}^{N} e^{u_k^\star(x_k)}\right) R^N(x_0,\cdots,x_N).
\end{equation}

The  algorithm is obtained by relaxations of the maximizations on the dual problem \eqref{dualEnt}. 
 We get the iterative method computing a sequence of potentials (denoted with the superscripts $.^{(n)}$)~:
 \\
 given $N+1$ vectors $u_k^{(0)}$ with $k=0,\cdots,N$, then the update at step $n$ is defined as
 \begin{equation}
 \label{iteratesDual}
 \displaystyle
 \begin{cases}
 u_k^{(n)}:=\argmax_u -\tilde F^\star(-u)-\int \exp(u)I^u_k\mbox{d}x_1\cdots\mbox{d}x_N\;    \quad \text{for}\;k=0,\\[10pt]

 u_k^{(n)}:=\argmax_u -\frac{T}{N}F^\star(-u)-\int \exp(u)I^u_k\mbox{d}x_0\cdots\mbox{d} x_{k-1} \mbox{d} x_{k+1} \cdots\mbox{d}x_N\;   \text{ for} \;k=1,\cdots,N-1,\\[10pt]
 
  u_k^{(n)}:=\argmax_u -G^\star(-u)-\int \exp(u)I^u_k\mbox{d}x_0\cdots\mbox{d}x_{N-1}\;   \quad \text{for} \;k=N,
   \end{cases}
 \end{equation}
 where 
 \[I^u_k:=\exp(\oplus_{i=0}^{k-1}u^{(n)}_i)\exp(\oplus_{i=k+1}^{N}u^{(n-1)}_i)R^N.\]

\begin{rem}[Examples of Energies ]
For many interesting energies $F$ and $G$,  the relaxed maximizations  can  be computed pointwise in space  
and analytically.   We  list here a few which are tested in the numerical section. 

\label{proxEx}
\begin{description}[align=left]
\item {\em Marginal constraint : } In this case the functional $F$ takes the form \\
$F(\rho)=i_{\rho_0}(\rho)$.

\item {\em Hard congestion : } For  the hard congestion one has \[F(\rho)=\begin{cases} 0\;\text{if}\; \rho \le \bar \rho,\\ +\infty\;\text{otherwise}.\end{cases}\]

\item {\em Potentials and  Obstacles}
In this case $F(\rho)  =  \int V(t,x)\, \rho(x) \, dx  $ is linear in $\rho$.   \\
  Obstacles  correspond to 
\[V(t,x) =\begin{cases} +\infty \;\text{if}\; x\in \Omega(t), \\  0 \;\text{otherwise}\end{cases}\]
where $t \mapsto \Omega(t)$ represent one or more  (possibly moving with time)  bounded domains.
\end{description}
\end{rem}

 \begin{rem}[Implementation] 
\label{implementation}
The iterations of Sinkhorn might seem tedious at a first glance because of the integration against $R^N$, but due to the special form of $R^N$ \eqref{decomposed}, these are just series of convolution with the kernel $P_t$. Moreover, the heat kernel can be further decomposed along the dimension as follows 
\[P_t(z)=\prod_{j=1}^d p_t(z_j)\] 
where $p_t(z_j)$ is the heat kernel in dimension one.  
This implies that instead of storing a matrix $P_t\in\R^{M\times M}$ (which is already better than the full tensor $R^N$) one can just store $d$ small matrices belonging to $\R^{\sqrt[d]{M}\times\sqrt[d]{M}}$.
\end{rem}

\subsection{Adding a viscosity parameter}


It is straightforward to extend the  theory and the numerical method to a slightly more general model with a viscosity parameter $\eps$. The Mean Field Game system
(\ref{MFGsystem2}) now  takes the form~:
\begin{equation}
\label{MFGsystem2eps}
\left\{\begin{array}{lll}
 -\partial_t u -\frac{\eps}{2} \Delta u  + \frac{1}{2} \vert \nabla u \vert^2= f[\rho_t], \; &(t,x)\in (0,T)\times \R^d\\
\partial_t \rho - \frac{\eps}{2} \Delta \rho -\dive(\rho \nabla u)=0, \; &(t,x)\in (0,T)\times \R^d \\
\rho{_{\vert}}_{t=0}=\rho_0, \; u{_{\vert}}_{t=T}= g[\rho_T].
\end{array}\right.\end{equation}
The Lagrangian formulation \eqref{minentrmfg} we have proposed becomes
 \begin{equation}\label{minentrmfgeps}
\inf \Big\{ H(Q\vert R_\eps)+ \CC((Q_t)_{t\in [0,T]}) \; : \; Q\in \PP(\Omega), \; Q_0=\rho_0\Big\}. 
\end{equation}
where $R_\eps$ is the reversible Wiener measure induced by a Brownian motion with variance $\eps$.
In particular, if we discretize the problem in time, we have that the reference measure $R_\eps^N$ can be still decomposed by using the heat kernel $P_{\eps t}(z)$.
Notice that we can still use the algorithm we have introduced in the previous section, but the performance, in terms of iterations to converge, will be affected by small values of $\eps$. At least formally, when the viscosity is small, \pref{minentrmfgeps} is an approximation of the following Lagrangian formulation of first-order variational mean-field games (see  \cite{benamou2017variational} for more details about this formulation)
 \begin{equation}\label{minimfg1stLag}
\inf \left\{ \KK(Q) +  \CC((Q_t)_{t\in [0,T]})  \; : \; Q\in\PP(\Omega),\;Q_0=\rho_0 \right\},
\end{equation}
where 
\begin{equation}\label{defdeK}
\KK(Q):=  \frac{1}{2}  \int_{\Omega}  \int_0^T \vert \dot \omega(t)\vert^2  \mbox{d}t \mbox{d} Q(\omega).
\end{equation}
This also  implies that we can use the Sinkhorn algorithm, with small $\eps$, in order to approximate the solution to first-order MFGs.
\begin{rem}[$\Gamma-$convergence]
We do not prove here $\Gamma-$convergence results as $\eps\rightarrow 0$, for the time discretization, it is indeed, at least on the flat torus, a straightforward generalization of a similar result obtained in \cite{benamou2017generalized} for the so-called Bredinger problem.
\end{rem}

\begin{rem}[Numerical limitations]
 Denoting by $h$ and $dt$  respectively  the space and time discretization steps, the heat kernel (\ref{heatker}) used in the Sinkhorn iterations scales like $e^{- \frac{h^2}{dt \, \eps}} $.
It induces  a limit to the stability of the numerical method  when letting $\eps$ go to 0 or increasing, respectively decreasing,  $h$ and $dt$. We refer to  \cite{schmitzer16} for a review of  several approaches to mitigate this problem in the 2-marginal case. 
\end{rem}

\subsection{Numerical Results}

\subsubsection{Planning MFG}

We illustrate  the effect of varying $\eps$  with the planning MFG problem on the torus. This correspond to $F:=0$ and $G := i_{\rho_1} $ (see remark \ref{proxEx}), the characteristic function 
which forces, as for the initial density,  the final density to match exactly a prescribed density. This problem is equivalent to the dynamic
formulation of  optimal transport \cite{benamou2000computational} which is a first-order planning MFG with the additional  expected diffusion effect linked to the second-order term. 

On figure \ref{fig:perio1},   for $\eps = 1$, one can see the effect of the diffusion and that mass will travel across the periodic boundary. On figures \ref{fig:perio2} ($\eps = 0.1$) and \ref{fig:perio3} ($\eps = 0.01$),  the transport gets  closer to the optimal transport/first-order MFG. 

\newcommand{\noObstacle}[2]{\includegraphics[width=.35\linewidth]{figures/cemracs#1-#2}}

\newcommand{\perio}[2]{\includegraphics[width=.35\linewidth]{figures/densperio#1#2}}
\newcommand{\operio}[2]{\includegraphics[width=.35\linewidth]{figures/densoperio#1-#2}}

\newcommand{\Obstacle}[2]{\includegraphics[width=.35\linewidth]{figures/Obstacle_hc_eps_#1_density_at_time_#2}}

\begin{figure}[h!]
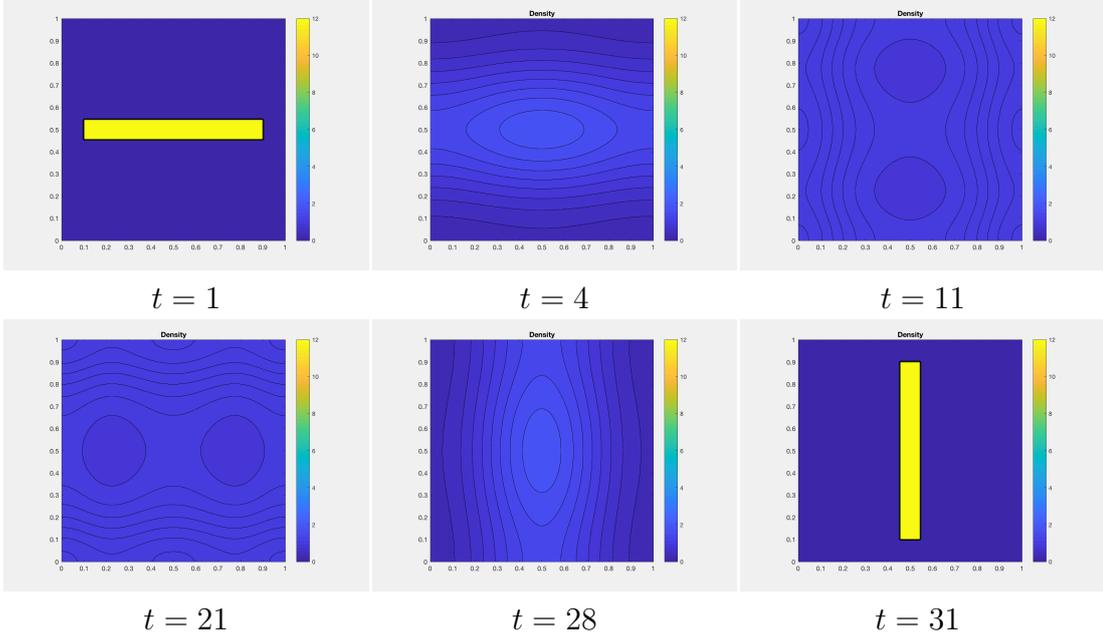


	\centering
       \TabThree{
       \perio{1}{0}&
       \perio{1}{1}&
       \perio{1}{2} \\
         $t=1$ & $t=4$ & $t=11$\\
      \perio{1}{3}&
       \perio{1}{4}&
       \perio{1}{5} \\
     
       $t=21$ &  $t= 28$ & $t=31$ 
       }
\caption{Planning MFG on the torus,    Densities at different time steps. $\eps=1$ and $31$ time steps.  }
\label{fig:perio1}
\end{figure}

\begin{figure}[h!]

	\centering
       \TabThree{
       \perio{2}{0}&
       \perio{2}{1}&
       \perio{2}{2} \\
         $t=2$ & $t=4$ & $t=11$\\
      \perio{2}{3}&
       \perio{2}{4}&
       \perio{2}{5} \\
     
       $t=21$ &  $t= 28$ & $t=30$ 
       }
\caption{Planning MFG on the torus, fixed initial and final densities similar to figure \ref{fig:perio1},   Densities at different time steps. $\eps=0.1$ and $31$ time steps.  }
\label{fig:perio2}
\end{figure}

\begin{figure}[h!]

	\centering
       \TabThree{
       \perio{3}{0}&
       \perio{3}{1}&
       \perio{3}{2} \\
         $t=2$ & $t=6$ & $t=13$\\
      \perio{3}{3}&
       \perio{3}{4}&
       \perio{3}{5} \\
     
       $t=17$ &  $t= 26$ & $t=30$ 
       }
\caption{Planning MFG on the torus, fixed initial and final densities similar to figure \ref{fig:perio1},   Densities at different time steps. $\eps=0.01$ and $31$ time steps.  }
\label{fig:perio3}
\end{figure}

\subsubsection{Moving Obstacles}

We now  use the same planning MFG setting but 
add multiple obstacles moving with time (see remark \ref{proxEx}).  
The boundaries of the obstacles are the white circles on the snapshots displayed in 
figures \ref{fig:operio1}-\ref{fig:operio3}.   The density is zero inside as agents pay an infinite cost to be there.

\begin{figure}[h!]

	\centering
       \TabThree{
       \operio{1}{0}&
       \operio{1}{1}&
       \operio{1}{2} \\
         $t=0$ & $t=4$ & $t=8$\\
      \operio{1}{3}&
       \operio{1}{4}&
       \operio{1}{5} \\
     
       $t=12$ &  $t= 16 $ & $t=20$  \\
             \operio{1}{6}&
       \operio{1}{7}&
       \operio{1}{8} \\
     
       $t=24$ &  $t= 28 $ & $t=32$ 

       }
\caption{Planning MFG on the torus with moving obstacles, fixed initial and final densities similar to figure \ref{fig:perio1},   Densities at different time steps. $\eps=1$ and $32$ time steps.  }
\label{fig:operio1}
\end{figure}

\begin{figure}[h!]

	\centering
       \TabThree{
       \operio{2}{0}&
       \operio{2}{1}&
       \operio{2}{2} \\
         $t=0$ & $t=4$ & $t=8$\\
      \operio{2}{3}&
       \operio{2}{4}&
       \operio{2}{5} \\
     
       $t=12$ &  $t= 16 $ & $t=20$  \\
             \operio{2}{6}&
       \operio{2}{7}&
       \operio{2}{8} \\
     
       $t=24$ &  $t= 28 $ & $t=32$ 

       }
\caption{Planning MFG on the torus with moving obstacles, fixed initial and final densities similar to figure \ref{fig:perio1},   Densities at different time steps. $\eps=0.1$ and $32$ time steps.  }
\label{fig:operio2}
\end{figure}

\begin{figure}[h!]

	\centering
       \TabThree{
       \operio{3}{0}&
       \operio{3}{1}&
       \operio{3}{2} \\
         $t=0$ & $t=4$ & $t=8$\\
      \operio{3}{3}&
       \operio{3}{4}&
       \operio{3}{5} \\
     
       $t=12$ &  $t= 16 $ & $t=20$  \\
             \operio{3}{6}&
       \operio{3}{7}&
       \operio{3}{8} \\
     
       $t=24$ &  $t= 28 $ & $t=32$ 

       }
\caption{Planning MFG on the torus with moving obstacles,  fixed initial and final densities similar to figure \ref{fig:perio1},   Densities at different time steps. $\eps=0.01$ and $32$ time steps.  }
\label{fig:operio3}
\end{figure}

\subsubsection{Non local interactions}\label{nonlocal}

In this section, we investigate a model for which $F$ is not convex in $\rho$. More precisely, the overall running cost will be of the form
  $F = F_1+F_2$, with  $F_1$ a hard congestion term preventing concentration of mass:
  \[F_1(\rho)=\begin{cases} 0\;\text{if}\; \rho \le  1, \\ +\infty\;\text{otherwise}. \end{cases}\] 
As for the $F_2$ term, it is given by a non local  interaction functional:
 \[ F_2(\rho) =   - \frac{1}{2} \int   \int  K(x-y) \rho( y)  \,  \mbox{d} y  \; \rho(x) \, \mbox{d} x.   \]
 If the kernel is non symmetric ($ K(x-y) \ne K(y-x))$ the mean field game is not variational, we only solve for a 
 fixed point of the system.  \\
 
The first kernel in  (see first frame figures \ref{PotepsSmall})   in polar coordinates,   is the tensor product of Gaussians in radius and direction. 
It is non symmetric and (because of the minus sign above) will favor streching the mass in the mean direction (here $45$ deg.). 
The initial density is prescribed and the final density is free. 
The kinetic energy cost penalizes displacement and the game seems to give a smooth deformation to a stationary  
optimal shape.  The periodic boundary conditions are not active as we use a low diffusion parameter $\eps$. \\

The second proposed kernel in figure  \ref{PotepsSmall2}  is the symmetrized version of the first.  
We now get a perfectly symmetric evolution toward a stationary shape.  \\

From an algorithmic point of view, we are out of the domain of application of Sinkhorn (\ref{iteratesDual}).  In order to solve the problem,  we adopt a semi-implicit approach (see \cite{carlierlaborde}). At each Sinkhorn  iteration $n+1$ we relax the non local functional as a potential linear cost using  the 
 density of the previous iteration:
  
 $f_2^{(n)}(x) =     - \int   K(x-y) \rho^{(n)} (y) \, dy$.

\newcommand{\Potpot}[2]{\includegraphics[width=.4\linewidth]{figures/P7#1-#2}}
\newcommand{\PotentialPot}[1]{\includegraphics[width=.4\linewidth]{figures/P7#1}}
\begin{figure}[h!]

	\centering
       \TabThree{
       \PotentialPot{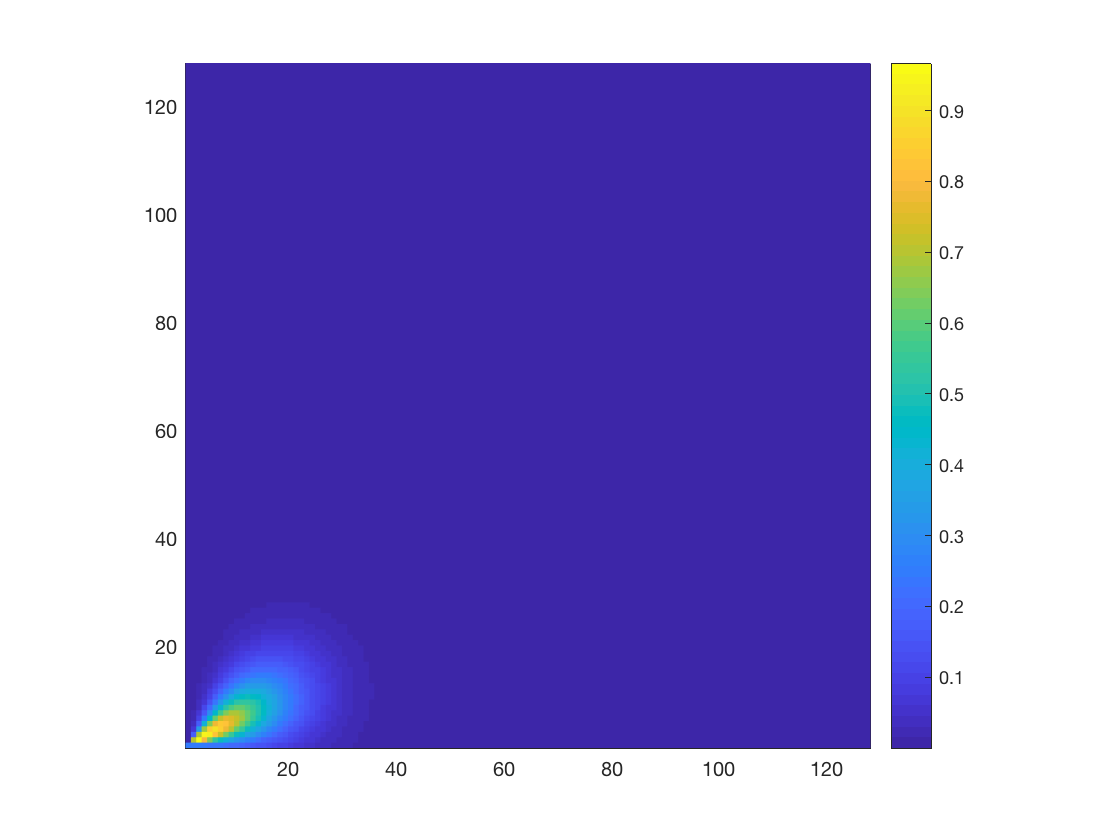} & \Potpot{4}{0} & \Potpot{4}{2} \\
       K(.)   & t = 1 & t =  11  \\
      \Potpot{4}{16} &       \Potpot{4}{30} &      \Potpot{4}{31}  \\
      t = 21 & t = 31  & t =  41 
       }
 \caption{Non local interaction. Top left : Kernel  $K$. Densities at different time steps. $\eps = 10^{-3} $.  } 
  \label{PotepsSmall}
\end{figure}

\begin{figure}[h!]

	\centering
       \TabThree{
       \PotentialPot{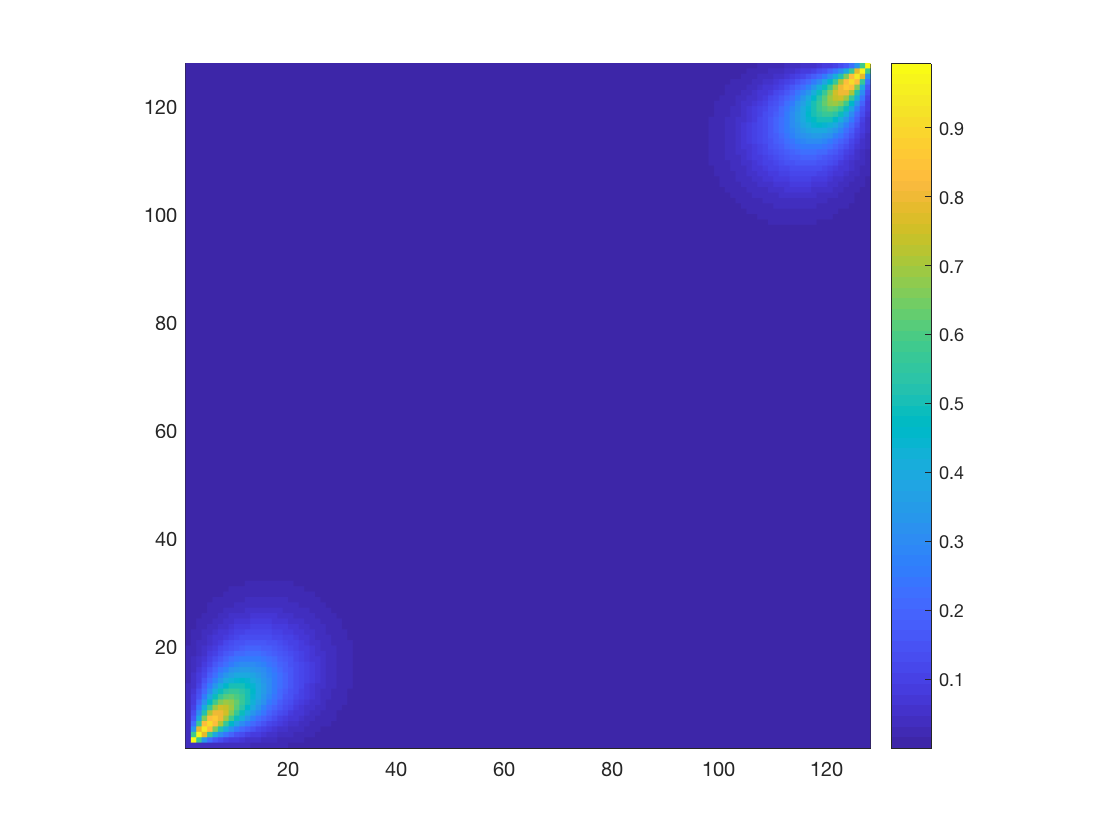}
     &\Potpot{3}{0}
    &\Potpot{3}{2}\\
        K(.)   & t = 1 & t =  11  \\
      \Potpot{3}{16} 
    &\Potpot{3}{30} &\Potpot{3}{31} \\
 t = 21 & t = 31  & t =  41 
       }
 \caption{Non local interaction. Top left : potential $K$. Densities at different time steps. $\eps = 10^{-3} $.  }
  \label{PotepsSmall2}
\end{figure}

\smallskip


\bibliographystyle{plain}

\bibliography{bibli}

\begin{thebibliography}{10}

\bibitem{achdou2}
Yves Achdou, Fabio Camilli, and Italo Capuzzo-Dolcetta.
\newblock Mean field games: convergence of a finite difference method.
\newblock {\em SIAM J. Numer. Anal.}, 51(5):2585--2612, 2013.

\bibitem{achdou1}
Yves Achdou and Italo Capuzzo-Dolcetta.
\newblock Mean field games: numerical methods.
\newblock {\em SIAM J. Numer. Anal.}, 48(3):1136--1162, 2010.

\bibitem{achdou3}
Yves Achdou and Alessio Porretta.
\newblock Convergence of a finite difference scheme to weak solutions of the
  system of partial differential equations arising in mean field games.
\newblock {\em SIAM J. Numer. Anal.}, 54(1):161--186, 2016.

\bibitem{AFP}
Luigi Ambrosio, Nicola Fusco, and Diego Pallara.
\newblock {\em Functions of bounded variation and free discontinuity problems}.
\newblock Oxford Mathematical Monographs. The Clarendon Press, Oxford
  University Press, New York, 2000.

\bibitem{AGS}
Luigi Ambrosio, Nicola Gigli, and Giuseppe Savar{\'e}.
\newblock {\em Gradient flows: in metric spaces and in the space of probability
  measures}.
\newblock Springer Science \& Business Media, 2006.

\bibitem{andreev}
Roman Andreev.
\newblock Preconditioning the augmented lagrangian method for instationary mean
  field games with diffusion.
\newblock {\em SIAM J. Sci. Comput.}, 39(6), 2017.

\bibitem{leonardBredinger}
Marc Arnaudon, Ana~Bela Cruzeiro, Christian L{\'e}onard, and Jean-Claude
  Zambrini.
\newblock An entropic interpolation problem for incompressible viscid fluids.
\newblock {\em arXiv preprint arXiv:1704.02126}, 2017.

\bibitem{benamou2000computational}
Jean-David Benamou and Yann Brenier.
\newblock A computational fluid mechanics solution to the monge-kantorovich
  mass transfer problem.
\newblock {\em Numerische Mathematik}, 84(3):375--393, 2000.

\bibitem{benamoucarlierjota}
Jean-David Benamou and Guillaume Carlier.
\newblock Augmented {L}agrangian methods for transport optimization, mean field
  games and degenerate elliptic equations.
\newblock {\em J. Optim. Theory Appl.}, 167(1):1--26, 2015.

\bibitem{Ben}
Jean-David Benamou, Guillaume Carlier, Marco Cuturi, Luca Nenna, and Gabriel
  Peyr{\'e}.
\newblock Iterative {B}regman projections for regularized transportation
  problems.
\newblock {\em SIAM Journal on Scientific Computing}, 37(2):A1111--A1138, 2015.

\bibitem{benamou2016numerical}
Jean-David Benamou, Guillaume Carlier, and Luca Nenna.
\newblock A numerical method to solve multi-marginal optimal transport problems
  with coulomb cost.
\newblock In {\em Splitting Methods in Communication, Imaging, Science, and
  Engineering}, pages 577--601. Springer International Publishing, 2016.

\bibitem{benamou2017generalized}
Jean-David Benamou, Guillaume Carlier, and Luca Nenna.
\newblock Generalized incompressible flows, multi-marginal transport and
  sinkhorn algorithm.
\newblock {\em Numerische Mathematik}, pages 1--22, 2017.

\bibitem{benamou2017variational}
Jean-David Benamou, Guillaume Carlier, and Filippo Santambrogio.
\newblock Variational mean field games.
\newblock In {\em Active Particles, Volume 1}, pages 141--171. Springer, 2017.

\bibitem{blanchet2016computation}
Adrien Blanchet, Guillaume Carlier, and Luca Nenna.
\newblock Computation of cournot--nash equilibria by entropic regularization.
\newblock {\em Vietnam Journal of Mathematics}, pages 1--17, 2016.

\bibitem{silva}
L.~M. Brice\~no Arias, D.~Kalise, and F.~J. Silva.
\newblock Proximal methods for stationary mean field games with local
  couplings.
\newblock {\em SIAM J. Control Optim.}, 56(2):801--836, 2018.

\bibitem{CGPT}
Pierre Cardaliaguet, P.~Jameson Graber, Alessio Porretta, and Daniela Tonon.
\newblock Second order mean field games with degenerate diffusion and local
  coupling.
\newblock {\em NoDEA Nonlinear Differential Equations Appl.}, 22(5):1287--1317,
  2015.

\bibitem{carlierlaborde}
Guillaume Carlier and Maxime Laborde.
\newblock A splitting method for nonlinear diffusions with nonlocal,
  nonpotential drifts.
\newblock {\em Nonlinear Anal.}, 150:1--18, 2017.

\bibitem{catleo}
Patrick Cattiaux and Christian L\'eonard.
\newblock Large deviations and {N}elson processes.
\newblock {\em Forum Math.}, 7(1):95--115, 1995.

\bibitem{CGP}
Yongxin Chen, Tryphon~T. Georgiou, and Michele Pavon.
\newblock On the relation between optimal transport and {S}chr\"odinger
  bridges: a stochastic control viewpoint.
\newblock {\em J. Optim. Theory Appl.}, 169(2):671--691, 2016.

\bibitem{chizat2016scaling}
Lenaic Chizat, Gabriel Peyr{\'e}, Bernhard Schmitzer, and Fran{\c{c}}ois-Xavier
  Vialard.
\newblock Scaling algorithms for unbalanced transport problems.
\newblock {\em arXiv preprint arXiv:1607.05816}, 2016.

\bibitem{Cut}
Marco Cuturi.
\newblock Sinkhorn distances: Lightspeed computation of optimal transport.
\newblock In {\em Advances in Neural Information Processing Systems}, pages
  2292--2300, 2013.

\bibitem{DawsonGartner}
Donald~A. Dawson and J\"{u}rgen G\"{a}rtner.
\newblock Large deviations from the {M}c{K}ean-{V}lasov limit for weakly
  interacting diffusions.
\newblock {\em Stochastics}, 20(4):247--308, 1987.

\bibitem{DellacherieMeyer}
Claude Dellacherie and Paul-Andr\'{e} Meyer.
\newblock {\em Probabilities and potential}, volume~29 of {\em North-Holland
  Mathematics Studies}.
\newblock North-Holland Publishing Co., Amsterdam-New York; North-Holland
  Publishing Co., Amsterdam-New York, 1978.

\bibitem{Follmer}
Hans F\"{o}llmer.
\newblock Random fields and diffusion processes.
\newblock In {\em \'{E}cole d'\'{E}t\'{e} de {P}robabilit\'{e}s de
  {S}aint-{F}lour {XV}--{XVII}, 1985--87}, volume 1362 of {\em Lecture Notes in
  Math.}, pages 101--203. Springer, Berlin, 1988.

\bibitem{galichon2015cupid}
Alfred Galichon and Bernard Salani{\'e}.
\newblock Cupid's invisible hand: Social surplus and identification in matching
  models.
\newblock 2015.

\bibitem{GLR}
Ivan Gentil, Christian L\'eonard, and Luigia Ripani.
\newblock About the analogy between optimal transport and minimal entropy.
\newblock {\em Ann. Fac. Sci. Toulouse Math. (6)}, 26(3):569--601, 2017.

\bibitem{GigTam18a}
Nicola Gigli and Luca Tamanini.
\newblock Benamou-{B}renier and duality formulas for the entropic cost on
  ${RCD}^*({K},{N})$ spaces, 2018.

\bibitem{gueant}
Olivier Gu\'eant.
\newblock Mean field games equations with quadratic hamiltonian: A specific
  approach.
\newblock {\em Mathematical Models and Methods in Applied Sciences},
  22(09):1250022, 2012.

\bibitem{MFGLL1}
Jean-Michel Lasry and Pierre-Louis Lions.
\newblock Jeux \`a champ moyen. {I}. {L}e cas stationnaire.
\newblock {\em C. R. Math. Acad. Sci. Paris}, 343(9):619--625, 2006.

\bibitem{MFGLL2}
Jean-Michel Lasry and Pierre-Louis Lions.
\newblock Jeux \`a champ moyen. {II}. {H}orizon fini et contr\^ole optimal.
\newblock {\em C. R. Math. Acad. Sci. Paris}, 343(10):679--684, 2006.

\bibitem{MFGLL3}
Jean-Michel Lasry and Pierre-Louis Lions.
\newblock Mean field games.
\newblock {\em Jpn. J. Math.}, 2(1):229--260, 2007.

\bibitem{leonard2012schrodinger}
Christian L{\'e}onard.
\newblock From the {S}chr\"{o}dinger problem to the {M}onge--{K}antorovich
  problem.
\newblock {\em Journal of Functional Analysis}, 262(4):1879--1920, 2012.

\bibitem{LeonardGirsanov}
Christian L\'{e}onard.
\newblock Girsanov theory under a finite entropy condition.
\newblock In {\em S\'{e}minaire de {P}robabilit\'{e}s {XLIV}}, volume 2046 of
  {\em Lecture Notes in Math.}, pages 429--465. Springer, Heidelberg, 2012.

\bibitem{leonard2013survey}
Christian L{\'e}onard.
\newblock A survey of the {S}chr\"{o}dinger problem and some of its connections
  with optimal transport.
\newblock {\em Discrete Contin. Dyn. Systems, A}, 34(4):1533--1574, 2014.

\bibitem{Mikami04}
Toshio Mikami.
\newblock Monge's problem with a quadratic cost by the zero-noise limit of
  {$h$}-path processes.
\newblock {\em Probab. Theory Related Fields}, 129(2):245--260, 2004.

\bibitem{thesislulu}
Luca Nenna.
\newblock {\em Numerical methods for multi-marginal optimal transportation}.
\newblock PhD thesis, PSL Research University, 2016.

\bibitem{peyre2015entropic}
Gabriel Peyr{\'e}.
\newblock Entropic approximation of wasserstein gradient flows.
\newblock {\em SIAM Journal on Imaging Sciences}, 8(4):2323--2351, 2015.

\bibitem{schmitzer16}
Bernhard Schmitzer.
\newblock {Stabilized Sparse Scaling Algorithms for Entropy Regularized
  Transport Problems}.
\newblock working paper or preprint, October 2016.

\bibitem{Schrodinger31}
Erwin Schr{\"o}dinger.
\newblock {\em {\"U}ber die umkehrung der naturgesetze}.
\newblock Verlag Akademie der wissenschaften in kommission bei Walter de
  Gruyter u. Company, 1931.

\bibitem{Sinkhorn67}
R.~Sinkhorn.
\newblock Diagonal equivalence to matrices with prescribed row and column sums.
\newblock {\em Amer. Math. Monthly}, 74:402--405, 1967.

\bibitem{Weissler}
Fred~B. Weissler.
\newblock Logarithmic {S}obolev inequalities for the heat-diffusion semigroup.
\newblock {\em Trans. Amer. Math. Soc.}, 237:255--269, 1978.

\end{thebibliography}

\end{document}